\documentclass{article}

\usepackage[a4paper, top=3cm, left=3cm, right=3cm, bottom=3.5cm]{geometry}
\usepackage{amsmath, amssymb}
\usepackage{enumerate}
\usepackage[backend=biber, style=alphabetic]{biblatex}
\usepackage{mathtools}
\usepackage{pdfpages}

\usepackage{color}
\usepackage{amsthm}
\usepackage{tikz}
\usepackage{asymptote}
\usepackage{hyperref}
\usepackage[labelfont={sc}]{caption}
\usepackage{tikz}
\usetikzlibrary{cd}

\renewbibmacro{in:}{}

\def\ltextindent#1{\hbox to \hangindent{#1\hss}\ignorespaces}
\long\def\ignore#1\recognize{}
\let\svthefootnote\thefootnote

\def\med{\medskip}
\def\hs{\hskip}

\def\cl{\centerline}
\def\ds{\displaystyle}
\def\ol{\overline}
\def\ul{\underline}
\def\sm{\setminus}

\def\{{\lbrace}
\def\}{\rbrace}

\def\map{\rightarrow}
\def\inv{^{-1}}
\def\abs#1{\vert#1\vert}
\def\wt{\widetilde}
\def\wh{\widehat}

\def\phi{\varphi}

\def\qed{\hfill$\square$\med}

\def\N{\mathbb{N}}
\def\R{\mathbb{R}}

\def\C{\mathbb{C}}
\def\Z{\mathbb{Z}}
\def\Q{\mathbb{Q}}
\def\kk{{\Bbbk}}

\def\K{{\Bbbk}}

\def\z{z}
\def\rr{\rho}
\def\SS_0{\sigma}
\def\zz{z}
\def\tt{s}
\def\tt{s}

\def\9{\6_\tt}
\def\DD{\partial}
\def\LL{L}
\def\L{L}
\def\TT{T}
\def\SS{S}
\def\FF{I\!\!F}
\def\FF{{\mathcal F}}

\def\HH{{\mathcal H}}

\def\O{{\cal O}}
\def\RR{{\cal R}}
\def\KK{{\cal K}}

\def\1{1\!\!1}
\def\kk{\Bbbk}

\newcommand{\6}{\partial}
\def\min{{\rm min}}
\def\max{{\rm max}}
\def\ord{{\rm ord}}
\def\Id{{\rm Id}}

\def\Ker{{\rm Ker}}
\def\Im{{\rm Im}}

\def\a{\alpha}

\renewcommand{\xi}{n_\rho}

\def\O{{\cal O}}
\def\RR{{\cal R}}
\def\F{{\mathbb{F}}}

\def\Quot{{\rm Quot}}

\def \osb{[\![}
\def \csb{]\!]}

\def \orb{(\!(}
\def \crb{)\!)}

\def \LL{{L}}

\theoremstyle{definition}
\newtheorem{defi}{Definition}[section]
\newtheorem{ex}[defi]{Example}

\theoremstyle{plain}
\newtheorem{thm}[defi]{Theorem}
\newtheorem{lem}[defi]{Lemma}
\newtheorem{cor}[defi]{Corollary}
\newtheorem{prop}[defi]{Proposition}
\newtheorem{conj}[defi]{Conjecture}
\newtheorem{problem}[defi]{Problem}
\theoremstyle{remark}
\newtheorem{rem}[defi]{Remark}

\begin{document}
\title{Fuchs' theorem on linear differential equations in arbitrary characteristic}
\author{Florian F\"urnsinn, Herwig Hauser}
\maketitle

\begin{abstract}

The paper generalizes Lazarus Fuchs' theorem on the solutions of complex ordinary linear differential equations with regular singularities to the case of ground fields of arbitrary characteristic, giving a precise description of the shape of each solution. This completes partial investigations started by Taira Honda and Bernard Dwork.

The main features are the introduction of a differential ring $\RR$ in infinitely many variables mimicking the role of the (complex) iterated logarithms, and the proof that adding these ``logarithms'' already provides sufficiently many primitives so as to solve {\em any} differential equation with regular singularity in $\RR$. A key step in the proof is the reduction of the involved differential operator to an Euler operator, its {\em normal form}, to solve Euler equations in $\RR$ and to lift their
(monomial) solutions to solutions of the original equation.

The first (and already very striking) example of this outset is the exponential function $\exp_p$ in positive characteristic, solution of $y' = y$. We prove that it necessarily involves all variables and we construct its explicit (and quite mysterious) power series expansion. Additionally, relations of our results to the Grothendieck-Katz $p$-curvature conjecture and related conjectures will be discussed.
\let\thefootnote\relax\footnote{ Supported by the Austrian Science Fund FWF, projects P-31338 and P-34765. We are indebted to A.~Bostan, M.~Singer, F.~Beukers, G.~Teschl, F.J.~Castro-Jim\'enez, L.~Narv\'aez, M.~Wibmer, C.~Chiu, C.~Schindler, N.~Merkl, M.~Reibnegger, F.~Lang, S.~Schneider
and S.~Yurkevich for very valuable input. The lively feedback of the participants of an online lecture series given by the second author in spring 2021 helped very much to shape the contents and the exposition of the present note.}
\addtocounter{footnote}{-1}\let\thefootnote\svthefootnote
\end{abstract}
\section{Introduction}
%Let $L=p_n\6^n+p_{n-1}\6^{n-1}+\ldots + p_1\6+p_0\in \O[\6]$ be a linear univariate differential operator with holomorphic or formal power series coefficients $p_i$ in $\O=\C\{x\}$, respectively, $\O=\K\osb x\csb$, $\K$ an arbitrary field. Write $L=\sum_{j=0}^n\sum_{i=0}^\infty c_{ij}x^i\6^j$ for its expansion at $0$, and denote by $L_0$ the {\it initial form} of $L$ at $0$, i.e., the Euler operator
%\[L_0=\sum_{i-j=\tau}^\infty c_{ij}x^i\6^j,\]
%where $\tau$ is the minimal {\it shift} $i-j$ occurring in the expansion. The {\it indicial polynomial} $\chi=\chi_L$ of $L$ at $0$ is defined as the polynomial $\chi(s)=\sum_{i-j=\tau} c_{ij}s^{\ul j}$ with $s^{\ul j}=s(s-1)\cdots (s-j+1)$, and its roots in $\C$, respectively in an algebraic closure $\ol \K$ of $\K$, are the {\it local exponents} of $L$ at $0$. We denote the multiplicity of $\rho$ as a zero of $\chi$ by $m_\rho$. Clearly, $L_0(x^k)=\chi(k) x^{k+\tau}$. If $L_0$ has the same order as $L$ we say that $L$ has a \textit{regular singularity} at $0$.

When solving the ordinary linear differential equation 
\begin{equation} \label{eq:deq}
 Ly=p_ny^{(n)}+p_{n-1}y^{(n-1)}+\ldots +p_1y'+p_0y=0
\end{equation}
with holomorphic coefficients and regular singularity at $0$, Fuchs' Theorem exhibits the following five phenomena:
\begin{enumerate}[(i)]
\item \textit{Euler equations} of the form
\begin{equation} \label{eq:euldeq}
c_nx^ny^{(n)}+c_{n-1}x^{n-1}y^{(n-1)}+\ldots +c_1y'+c_0y=0,
\end{equation}
for $c_i\in \C$, require monomial solutions $x^\rho\coloneqq\exp(\rho \log x)$ where $\rho\in \C$ are the local exponents of the equation at the singular point $x=0$. A full basis of solutions is given by $x^\rho \log(x)^i$, where $0\leq i<m_\rho$ and $m_\rho$ denotes the multiplicity of $\rho$ as a root of the indicial polynomial.
\item For each local exponent $\rho$, maximal among the local exponents congruent to each other modulo $\Z$, one can solve the linear recursion given by the equation iteratively. One obtains a  solution $y$ of \eqref{eq:deq} lifting the solution $x^\rho$ from $L_0y=0$, where $L_0$ denotes the initial form of $L$
\[y=x^\rho \sum_{i=0}^\infty a_ix^i \]
with $a_i\in\C$ and $a_0=1$.
\item If $\rho$ in addition is a multiple root, all solutions $x^\rho \log(x)^i$ of the initial equation can be lifted to solutions of \eqref{eq:deq}. These solutions are then of the form
\[ y_k=\sum_{i=0}^k h_{ki}(x)\log(x)^i,\]
for $k=0,\ldots, m_\rho-1$, with power series $h_{ki}$.
\item The next phenomenon is ``resonance''. That is the case when some local exponents differ by an integer. Let $\Omega=\{\rho_1,\ldots, \rho_s\}$ be such a set of exponents, ordered by $\rho_1>\rho_2>\ldots>\rho_s$ via $\rho_i-\rho_{i+1}\in \N$, and let $m_i$ be the multiplicity of $\rho_i$. Here, the lifting of the monomial solutions  $x^\rho\log(x)^i$ is more challenging, as higher orders of the logarithm appear. The solutions are of the form
\[ y_{jk}=x^{\rho_j}\sum_{i=0}^{\mathclap{m_{1}+\ldots +m_{j-1}+k}} h_{ijk}(x)\log(x)^i,\]
for $j=1,\ldots, s$ and $k=0,\ldots, m_j-1$.
\item Finally, the convergence of the coefficients $h_{ijk}$ is achieved via norm estimates on their coefficients. An extra issue is the possible algebraicity of the solutions (not involving logarithms). This classical problem, solved algorithmically by Singer in~\cite{Singer79}, is prominently presented by the still unsolved Grothendieck-Katz $p$-curvature conjecture. 

%\item For non-maximal local exponents $\rho$ let $\rho_1,\ldots, \rho_k$ be the other local exponents of \eqref{eq:deq} with $\rho_i-\rho\in \N$. Then it is possible to lift the solutions $x^\rho \log(x)^i$ of the initial equation to solutions of \eqref{eq:deq} and one obtains solutions in which the logarithm may appear with powers up to $m_\rho+m_{\rho_1}+\ldots +m_{\rho_k}-1$, i.e., solutions of the form
%\[ y=x^\rho\sum_{j=0}^{\mathclap{m_{\rho_1}+\ldots +m_{\rho_k}+i}} h_{ij}(x)\log(x)^j,\]
%for $i=0,\ldots, m_\rho-1$. Doing this for all local exponents at once, one obtains a full basis of $n=\ord L$ $\C$-linearly independent solutions. 
%\item Finally, a majorisation argument shows that all the coefficient series $h_{ij}$ are convergent, i.e., define a holomorphic function in a neighbourhood of $0$. The question of algebraicity for these solutions (especially the ones without logarithms) is a famous and notorious problem, culminating in the Grothendieck $p$-curvature conjecture.
\end{enumerate}
The collection of these facts is known as Fuchs' Theorem \cite{Fuchs66}, with a simplified proof given soonafter by Frobenius~\cite{Frobenius73}.
 
%\begin{ex}
%\begin{enumerate}[(a)]
%\item Consider the Euler operator \[L=x^3\partial^3-4x^2\partial^2+9x\partial-9.\]
%Its indicial polynomial reads  $\chi(s)=(s-3)^2(s-1)$ and thus its local exponents are $3$ with multiplicity $2$ and $1$ with multiplicity $1$. A basis of solutions of $Ly=0$ therefore is $x, x^3, x^3\log x$. 
%\item The equation $y'=y$ is satisfied by the exponential function. Indeed, making an ansatz $y=\sum_{k=0}^\infty a_kx^k$ one obtains the recursion $k a_k=a_{k-1}$, which is satisfied by the sequence $a_k=1/k!$ and gives the series expansion of the exponential function.
%\end{enumerate}
%\end{ex}

In a field $\kk$ of characteristic $p>0$ the situation is considerably more delicate: Instead of holomorphic coefficients we now allow polynomial or formal power series coefficients. 
%The local exponents will lie in the algebraic closure $\overline{\kk}$ of the the ground field, such that we need to consider ``monomials'' $t^\rho$. The new letter $t$ is introduced to separate elements of $\rho\in \overline{\F_p}$ from $k\in \Z$ in the exponents.  The elements $t^\rho$ are thus defined as symbols subject to the differentiation rule
%\[\partial t^\rho=\frac \rho x t^{\rho}.  \] Their algebraic relations are
%\[t^\rho\cdot t^\sigma = t^{\rho+\sigma}.\]
%Note that $\log x$ is no longer defined as a function. 
Honda \cite{Honda81} introduced a new variable $z$ with $\partial z = 1/x$ as a formal logarithm and studied when this extension suffices to describe solutions in $\kk(z)\orb x\crb$. Dwork generalized the ideas of Honda and treated the case of differential equations with nilpotent $p$-curvature in \cite{Dwork90}. Introducing finitely many new variables $z_1,\ldots, z_k$ instead of one variable $z$ with the differentiation rule $z_i'=z_{i-1}'/z_{i-1}$ for $i\geq 2$ and $z_1'=1/x$, he showed that \eqref{eq:deq} has nilpotent $p$-curvature if and only if it admits a basis of solutions in $\kk(x, z_1,\ldots, z_k)$ over the field of constants $\kk(x^p, z_1^p,\ldots, z_k^p)$. Note that the nilpotence of the $p$-curvature implies that \eqref{eq:deq} has a regular singularity at $0$ and that all local exponents lie in the prime field~\cite[Thm.~1, Prop.~2.1]{Honda81}.\\

The present paper continues this program and gives a full and detailed version of Fuchs' Theorem in the case of characteristic  $p>0$ and for all equations with regular singularity. We describe explicitly full bases of solutions of \eqref{eq:deq}. While in characteristic $0$ the exponents $k\in\Z$ of genuine power series lie in the algebraic closure $\overline \kk$ of the ground field $\kk$, this is no longer the case in positive characteristic. Therefore, monomials of the form $x^{\rho+k}$ with $\rho\in\overline\kk$ and $k\in\Z$ do not make sense without further specification. We dissolve this drawback by introducing a new variable $t$ and symbols (say, monomials)

$$t^\rho x^k$$

for $\rho\in \overline{\kk}, k\in\Z$, with multiplication $t^\rho x^k\cdot t^\sigma x^\ell=t^{\rho +\sigma} x^{k+\ell}$ and differentiation rule

$$\6(t^\rho x^k)= (\rho + \overline{k})t^\rho x^{k-1}.$$

In particular, $\6(t^\rho)=\frac{\rho}{x}t^\rho$ (this choice is preferable to the other option $\6(t^\rho)=\rho t^{\rho-1}$). Here, $\overline k$ denotes the residue of $k$ in the prime field $\F_p$ of $\overline \kk$.

The appropriate differential ring to work with is 
$$\mathcal{R}= \bigoplus_{\rho \in \overline{\kk}}t^\rho \kk(z_1, z_2, \ldots )\orb x\crb$$
with the derivation rules described above. Its ring of constants is 
$$\mathcal{C}:=\bigoplus_ {\rho \in \F_p}t^\rho x^{p-\rho}\kk(\z^p)\orb x^p\crb,$$
where $\F_p\subseteq \kk$ denotes the prime field. The five instances from above reappear again and translate as follows:
\begin{enumerate}[(i)]
\item To solve Euler equations \eqref{eq:euldeq} we define for $i\in \N$ the tuple
\[i^*=\left(i, \left \lfloor \frac i p \right \rfloor, \left \lfloor \frac i {p^2} \right \rfloor,\ldots\right) \in \N^{(\N)}, \]
in  which only finitely many entries are non-zero. Then the monomials $t^\rho z^{i^*}$, for $\rho$ a local exponent of multiplicity $m_\rho$, and $0\leq i<m_\rho$,  form a $\mathcal{C}$-basis of solutions of \eqref{eq:euldeq}.
\item The condition on $\rho$ to be maximal among the local exponents congruent to each other modulo $\Z$ is void. Further, it is no longer possible to construct the coefficients of a solution iteratively by the recursion determined by the differential equation in characteristic $p$:
%: If $\rho \in \overline \kk$ is a root of the indicial polynomial $\chi$ of \eqref{eq:euldeq} with monomial solution $t^\rho$ of the initial equation $L_0y=0$ then $t^\rho x^p$ is a solution as well. 
When making an unknown ansatz to find a power series solution of \eqref{eq:deq} in $t^\rho \overline{\kk}\osb x\csb$, one can in general solve the recursion up to order $p-1$ before arriving at a definite stop, due to an unsolvable equation for the next coefficient. We will show that the obstruction can be overcome when solving the equation in the extension $t^\rho\overline \kk(z_1, z_2,\ldots) \orb x\crb$. 
\item[(iii, iv)] Again it is impossible to distinguish between maximal and non-maximal local exponents. But the same procedure as used to construct a solution in (ii) allows us to ``lift'' the solutions $t^\rho z^{i^*}$ of $L_0y=0$ to solutions of \eqref{eq:deq} in $t^\rho\overline \kk(z_1, z_2,\ldots) \orb x\crb$.
\item[(v)] Convergence of the solutions is no longer an issue in characteristic $p$. 
However, one may ask whether there exist solutions which are {\em algebraic} (or even polynomials). This has been proven by Honda and Dwork for solutions involving only finitely many $z_i$ variables. But for infinitely many $z_i$, the question is more intriguing, and the exponential function is the first candidate to explore. This issue will be investigated in a later section of the paper through the projections of the solutions obtained by setting all but finitely many $z_i$'s equal to zero.
\end{enumerate}

\begin{ex}
\begin{enumerate}[(a)]
\item Consider again the Euler operator \[L=x^3\partial^3-4x^2\partial^2+9x\partial-9.\]
In all characteristics $p>2$ the equation $Ly=0$ has the basis of solutions solutions $t^1$, $t^3$, $t^3z_1$. In characteristic $2$ a basis of solutions is $t^1, t^1z_1, t^1z_1^2z_2$.
\item The equation $y'=y$ is one of the simplest differential equations with non-nilpotent $p$-curvature. In characteristics $2,3$ and  $5$ one finds the solutions 
\begin{align*}
\exp_2&=1+x+x^2z_1+x^3(z_1+1)+x^4(z_1^2z_2+z_1)+x^5z_1^2z_2+x^6(z_1^3z_2+z_1^3)\\ &\qquad+x^7(z_1^3z_2+z_1^2z_2+z_1^3+z_1+1)+\ldots,\\
\exp_3&=1+x+2x^2+2x^3z_1+x^4(1+2z_1)+x^5z_1+2x^6z_1^2+x^7(1+2z_1+2z_1^2)\\ &\qquad+x^8(2+z_1^2)+x^9(2z_1+z_1^3z_2)+\ldots,\\
\exp_5&=1+x+3x^2+x^3+4x^4+4x^5z_1+x^6(4z_1+1)+x^7(2z_1+2)\\ &\qquad +x^8(4z_1+1)+x^9z_1+3x^{10}z_1^2+\ldots
\end{align*}
\end{enumerate}
These solutions are chosen in such a way that no monomial which is a $p$-th power (except for $1$) appears in the series representation. In future work of the authors with Kawanoue it will be proven that the projections $\exp_p\vert_{z_i=0}$ are \emph{algebraic} over $\F_p(x, z_1,\ldots, z_{i-1})$ for all $i$.
\end{ex}

The program sketched in items (i) to (v) above is realized via a normal form theorem for linear differential operators. Reducing the operator in a controlled and explicit way to an Euler operator establishes the possibility to solve the associated equation by going backwards, lifting the solutions of the Euler equation to solutions of the original equation. Let us briefly explain this approach.

We start by transcribing Fuchs' Theorem in characteristic zero to the following (stronger) statement: Let $L$ be a differential operator with initial form $L_0$. On a suitable function space $\mathcal{F}$ on which $L$ acts, and which consists of (holomorphic) power series and powers of the complex logarithm up to a certain power determined by the local exponents, the operator $L$ can be transformed into its initial form $L_0$ under a linear automorphism of $\mathcal{F}$. In other words, the initial operator is a normal form of the differential operator up to the action of linear automorphisms on $\mathcal{F}$. The automorphism $u$ of $\FF$ will be explicitly constructed. In this way it provides an algorithm to compute the series expansion of a basis of solutions of the equation $Ly=0$ by pulling back via $u^{-1}$ the (obvious) solutions of $L_0y=0$.

This statement translates to positive characteristic: For any differential operator over a field $\kk$ of characteristic $p>0$ we define a subspace $\mathcal{F}\subseteq \mathcal{R}$ of the differential ring $\mathcal{R}$ described above and again, we can explicitly construct an automorphism of $\mathcal{F}$, which transforms $L$ into $L_0$. Also in the case of positive characteristic this will provide an algorithm to compute a basis of solutions of $Ly=0$.

{\bf Structure of the paper.} Section \ref{subsec:constr0} starts with a review of univariate differential operators, the definition of their initial form, the indicial polynomial and the local exponents. We describe the solutions of Euler equations over the complex numbers and introduce the differentiation of a differential operator with respect to the exponents (in the sense of Frobenius). Then, in Section \ref{subsec:nft0}, we construct a function space $\FF$ for equations in zero characteristic and then prove the respective normal form theorem, Theorem \ref{thm:nft0}. This provides in Section \ref{subsec:sol0} the description of a full basis of solutions in case the origin is a regular singular point, see Theorem \ref{thm:sol0}.
% For irregular singularities, we sketch in Section \ref{subsec:app0} Merkl's algorithm of how to use the normal form theorem also in this case to obtain all solutions. The section also includes a brief discussion of the occurrence of apparent singularities and of Gevrey series in this context. 

Chapters \ref{sec:deqp} and \ref{sec:appp} are devoted to positive characteristic. We start with the construction of primitives, the enhanced enlargement of function spaces, and the respective ring of constants (Sections \ref{subsec:lackp} and \ref{subsec:closurep}). These techniques are applied in section \ref{subsec:eulerp} for solving Euler equations in characteristic $p$. Section
\ref{subsec:nftp} contains the normal form theorem in positive characteristic, Theorem \ref{thm:nft}, together with its proof. This is then applied in section \ref{subsec:solp} to construct the associated solutions of differential equations with regular singularity (now defined through the order condition on the coefficients as given by Fuchs' criterion in characteristic $0$), in Theorem \ref{thm:solp}. 

With section \ref{subsec:exp} we begin to look at concrete examples as are the exponential function and the logarithm in characteristic $p$. In Section \ref{subsec:specialp} we study the case when only finitely many variables $z_i$ are needed to solve the equations, and relate this to the nilpotence of the $p$-curvature as described by Dwork. Also we ask and answer the question when the differential equation has even polynomial solutions, thus generalizing a result of Honda.

Section \ref{subsec:pcurv} compares the two normal form theorems with Grothendieck's $p$-curvature conjecture as well as with B\'ezivin's conjecture. The delicacy lies in the fact that the algorithm provided by the normal form theorem in positive characteristic is not the reduction modulo $p$ of the characteristic $0$ algorithm. The difference is subtle, and we aim at highlighting the involved phenomena (some of them being of purely number theoretic flavor). The article concludes in Section \ref{subsec:outlook} with the discussion of the integrality of the solutions, i.e., the question when the solutions of differential equations defined over $\Z$ have integer coefficients.

%---------------------------------

\section{Differential equations in characteristic zero} \label{sec:deq0}

\subsection{Constructions with differential operators} \label{subsec:constr0}

{\bf Singular differential equations.} 
Let be given a linear ordinary differential equation 
\[Ly=p_n(x) y^{(n)}+p_{n-1}(x)y^{(n-1)}+\ldots + p_1(x)y'+p_0(x)y=0,\]
where 
\[L=p_n\6^n+p_{n-1}\6^{n-1}+\ldots + p_1\6+p_0\in\O[\6]\]
is a differential operator. Here $\mathcal{O}$ denotes denotes the ring of  germs of holomorphic functions in one variable $x$ at a given chosen singular point of $L$, say, the origin $0$, or the ring of polynomials $\kk[x]$ or formal power series $\kk\osb x\csb$ over an arbitrary field $\kk$ of any characteristic. Moreover, $\6={\frac{d}{dx}}$ denotes the usual derivative with respect to $x$. Writing $L=\sum_{j=0}^n \sum_{i=0}^\infty c_{ij}x^i\6^j$, the operator decomposes into a sum 
\[L=L_0+L_1+\ldots +L_m+ \ldots\]
of {\it homogeneous} or {\it Euler operators} $L_k =\sum_{i-j=\tau_k}c_{ij}x^i\6^j$, where the {\it shifts} $\tau_0<\tau_1<\ldots$ of the operators $L_k$ are ordered increasingly and all $L_k$ are assumed to be non-zero. The term $L_0$ of smallest shift constitutes the {\it initial form} of $L$ at $0$, and $\tau:=\tau_0$ is called the shift of $L$ at $0$. Up to multiplying $L$ with the monomial $x^{-\tau}$ we may assume (as we will do throughout) that $L$ has shift $\tau=0$; thus $L_0=\sum_{i=0}^n c_{ii}x^i\6^i$. The point $x=0$ is singular for $L$ if at least one quotient $p_i/p_n$ has a pole at $0$ (otherwise, $0$ is called non-singular or ordinary). It is a {\it regular singularity} (in the sense of Fuchs) if $L_0$ has again order $n$, i.e., if $c_{nn}\neq 0$. %An operator $L$ with holomorphic coefficients in $\P^1_\C$ with at most regular singularities is called {\it Fuchsian}. 
The {\it indicial polynomial} of $L$ at $0$ is defined as
\[\chi_L(\tt)=\sum_{i=0}^n c_{ii} \tt^{\ul i}
=\sum_{i=0}^n c_{ii} \tt(\tt-1)\cdots (\tt-i+1).\]
Here, $\tt^{\ul i}$ denotes the falling factorial or Pochhammer symbol. Clearly, $\chi_L=\chi_{L_0}$, which we simply denote by $\chi_0$. Its roots $\rho$ in the algebraic closure $\overline{\kk}$ of $\kk$ are  the {\it local exponents} of $L$ at $0$, and $m_\rho\in\N$ will denote their multiplicity. 

\begin{rem} \label{rem:stirling}
(i) We can rewrite any differential operator in terms of $\delta:=x\partial$, the {\it Euler derivative}. The base change between $x^n\partial^n$ and $\delta$ is given by the Stirling numbers of the second kind $S_{n,k}$. This is readily verified using the recursion relation $S_{n+1,k}=kS_{n,k}+S_{n, k-1}$. This allows one to read off the indicial polynomial of an operator: If the initial form of an operator $L$ is given by $L_0=\phi(\delta)$ for some polynomial $\phi$, then the indicial polynomial of the operator is $\chi_L=\phi$.

(ii) The classical characteristic zero definition of a regular singular point of a differential equation using the growth of the local solutions cannot be translated to characteristic $p$. However, the equivalent characterization by Fuchs using the order of vanishing of the coefficients of the equation applies.
\end{rem}

We recall some basic facts from differential algebra. If $(R,\partial)$ is a differential ring (or field), a {\it constant} is an element $r\in R$, such that $\partial r=0$. The set of constants of $R$ forms a subring (or subfield). A linear differential equation of order $n$ has at most $n$ linearly independent solutions in any differential field $R$ over its field of constants. This is a simple corollary of Wronski's lemma, see \cite{SvdP03}, p.~9, or \cite{Honda81}. A set of $n$ linearly independent solutions is called a {\it full basis of solutions} of the equation in $R$. In particular, if $L\in \mathcal{O}[\partial]$ is a differential operator with holomorphic coefficients, then $Ly=0$ can only have $n$ $\C$-linearly independent solutions in $\mathcal{O}(\log x )$.

From now on we stick to characteristic $0$ and let $\mathcal{O}$ be the ring of germs of holomorphic functions at $0$.

%---------------------------------
%             SOLUTIONS EULER EQUATIONS
%---------------------------------

{\bf Euler equations.} The solutions of Euler equations $L_0y=0$ are easy to find. They are of the form 
\[y_{\rho,i}=x^\rho\log(x)^i,\]
where $\rho\in\C$ is a local exponent and $i$ varies between $0$ and $m_\rho-1$. Here, $x^\rho=\exp(\rho\log x)$ and $\log x$ may be considered either as a symbol subject to the differentiation rule $\6 x^\rr =\rr x^{\rr-1}$ and $\6 \log x= 1/x$, or as a holomorphic function on $\C_{\text{slit}}=\C\sm \R_{\geq 0}$ or on arbitrary simply connected open subsets of $\C^*=\C\sm\{0\}$. 
%One objective of the present paper is to lift these obvious solutions of $L_0y=0$ to solutions of the original equation $Ly=0$.

%---------------------------------
%             EXTENSIONS
%---------------------------------

{\bf Extensions of differential operators.} The consideration of logarithms is best formalized by introducing a new variable $z$ for $\log x$ \cite{Honda81}, \cite{Mezzarobba11}. To this end, equip the polynomial ring $\KK [\zz]$ over the field $\KK=\Quot(\O)$ of meromorphic functions at $0$ with the $\C$-derivation
\begin{gather*}
\DD:\KK [\zz]\map \KK [\zz],\\
\DD x =\6 x=1, \hs .6cm \DD\zz =x\inv,\\
%so that
\DD(x^i\zz^k) =(i\zz+k)x^{i-1}\zz^{k-1}.
\end{gather*}
This turns $\KK[z]$ into a differential ring. It carries in addition the usual derivative $\6_z$ with respect to $z$. The same definition applies to $\O x^\rho[z]$ for any $\rho\in\C$, taking $\DD x^\rho=\rho x^{\rho-1}$.

\begin{rem}
In $\mathcal{K}[z]$ every element has a primitive; it is the smallest extension of $\mathcal{K}$ for which this holds true. Indeed, $x^{-1}$ has no primitive in $\mathcal{K}$. The primitive of $x^{-1}z^\ell$ in $\mathcal{K}[z]$ is given by $\frac{1}{\ell+1}z^{\ell+1}$, while the primitive of $x^kz^\ell$ for $k\neq -1$ is given by $x^{k+1}p(z)$, where $p$ is a polynomial of degree $\ell$. Thus we may call $\mathcal{K}[z]$ the {\it primitive closure} of $\mathcal{K}$.
\end{rem}

The $j$-fold composition $\DD\circ\cdots \circ \DD$ will be denoted by $\DD^j$. For a differential operator $L=p_n\6^n+p_{n-1}\6^{n-1}+\ldots + p_1\6+p_0\in\O[\6]$ define its {\it extension} as the induced action on $\mathcal{K}[z]$, denoted by the same letter,
\[\LL:\mathcal{K}[z]\to \mathcal{K}[z]\]
If $\rho\in\C$ is a local exponent of $L$, we will likewise associate to $\LL$ the $\C$-linear map
\[\LL:\KK x^\rho[z]\map \KK x^{\rho}[z],\, x^\rho h(x) z^i\mapsto \LL(x^\rho h(x) z^i),\]
called again the extension of $L$ to $\KK x^\rho[z]$. Whenever $L$ has shift $\tau\geq 0$ -- as we will assume in the sequel -- its extension sends $\O x^\rho[z]$ to $\O x^\rho[z]$ and thus defines a $\C$-linear map 
\[\LL:\O x^\rho[z]\map \O x^{\rho}[z],\, x^\rho h(x) z^i\mapsto \LL(x^\rho h(x) z^i).\]
The Leibniz rule gives 

%-------------------------------------------------
%       LEMMA 1 SUBSTITUTION OF log(x) FOR z
%-------------------------------------------------

\begin{lem} \label{lem:substitutionlog}
Let $L$ be an operator. Then, for $\rho\in\C$, $h\in\O$, and $i\geq 0$,
\[\LL(x^\rr h(x)\zz^i)_{\vert \zz=\log x }=L(x^\rr h(x)\log(x)^i),\]
where on the right hand side $L$ acts via $\frac{d}{dx}$. In particular, the map $\O x^\rho [\zz]\map \O x^\rho [\log x ]$ given by the evaluation $z\mapsto \log x $ sends solutions of $\LL y=0$ to solutions of $Ly=0$.
\end{lem}
%---------------------------------
%             EXAMPLE EULER
%---------------------------------

\begin{ex} The equation $x^2y''+3xy'+1=0$ with Euler operator $L_0=x^2\6^2+3x\6+1$ has indicial polynomial $\chi_0=\rho^{\ul 2}+3\rho^{\ul 1}+1=(\rr+1)^2$ with double root $\rr=-1$. The solutions of $L_0y=0$ are $y_1=x\inv$ and $y_2=x\inv\log x $. The operator $\LL_0=x^2\DD^2+3x\DD+1$ on $\O x\inv[z]$ therefore has, as it should be, solutions $x\inv$ and $x\inv z$. Indeed, $\LL_0(x\inv)=L_0(x\inv)=0$, whereas $\DD(x\inv z)=x^{-2}(-z+1)$ and 
\[\DD^2(x\inv z)=\DD(x^{-2}(-z+1))= -2 x^{-3}(-z+1) - x^{-3}=x^{-3}(2z-3)\]
give
\[\LL_0(x\inv z) = x\inv (2z-3) +3 x\inv (-z+1) +x\inv z= 0.\] %x\inv(2z-3-3z+3+z)=0$.}
\end{ex}
%---------------------------------
%             FUNCTION SPACES
%---------------------------------

{\bf Function spaces.} If $L_0$ is an Euler operator with exponents set $\Omega\subseteq \C$ and if $m_\rho$ denotes the multiplicity of $\rho\in\Omega$, the $\C$-vector space
\[\FF_0=\sum_{\rho\in\Omega} \O x^\rho[z]_{<m_\rho}\]
of polynomials in $z$ of degree $<m_\rho$ and with coefficients in $\O x^\rho$ is the correct space to look at for finding the solutions of the extended Euler equation $\LL_0y=0$, since these are of the form $x^\rho z^i$, for $\rho\in\Omega$ and $0\leq i<m_\rho$. The space $\FF_0$ is, however, in general too small to contain the solutions of the extension $\LL y=0$ if $Ly=0$ is a general equation with regular singularity and initial form $L_0$. A suitable enlargement of $\FF_0$ is necessary. The method how to do this goes back to Fuchs, Frobenius, and Thom\'e; it requires some preparation.

%---------------------------------
%             DIFFERENTIATION of \6
%---------------------------------

{\bf Differentiating differential operators.} This technique first appears in the works of Frobenius. If $s$ is another variable, write the $j$-th derivative of $x^s=\exp(s\log x )$ as $\6^jx^s=s^{\ul j}x^{s-j}$. Define then, for $\ell\geq 1$, the {\it $\ell$-th derivative} $(\6^j)^{(\ell)}$ of $\6^j$ as
\[(\6^j)^{(\ell)} x^s=(s^{\ul j})^{(\ell)}x^{s-j},\]
where $(s^{\ul j})^{(\ell)}$ denotes the $\ell$-th derivative of $s^{\ul j}$ with respect to $s$. Clearly, $(\6^j)^{(\ell)}=0$ for $\ell>j$.  
Then, for a differential operator $L=p_n\6^n+p_{n-1}\6^{n-1}+\ldots + p_1\6+p_0$ of order $n$, we get its $\ell$-th derivative $L^{(\ell)}$ for $\ell\geq 1$ as
\[L^{(\ell)}=p_n\cdot (\6^n)^{(\ell)}+p_{n-1}\cdot (\6^{n-1})^{(\ell)}+\ldots + p_{1}\cdot(\6)^{(\ell)}.\]
This is no longer a differential operator; it is just a $\C$-linear map $\O t^\rho\map \O x^{\rho+\tau}$, where $\tau$ is the shift of $L$.

%---------------------------------
%             REMARK DIFFERENTIATION IN POSITIVE CHAR
%---------------------------------
\ignore

{\bf Remark.} If we wish to work in arbitrary characteristic, it might be appropriate to define the {\it $\ell$-th derivative} $(\6^j)^{[\ell]}$ of $\6^j$ differently as
\[\ds (\6^j)^{[\ell]} x^t={1\over \ell!} (t^{\ul j})^{(\ell)}x^{t-j}=\Delta^\ell(t^{\ul j})x^{t-j},\]
where $\Delta^\ell(t^k)$ is defined as the {\it divided or Hasse derivative} of $t^k$,
\[\Delta^\ell(t^k)={k\choose \ell} t^{k-\ell}.\]

\recognize

%---------------------------------
%           START LEMMATA 2 - 3  
%---------------------------------

The following facts are readily verified, cf. Lemmata \ref{lem:66mon} and \ref{lem:monomials} for similar results in positive characteristic. Let $L$ always be a differential operator of order $n$ and shift $\tau\geq 0$. Let $\rho\in \C$ be arbitrary.

%---------------------------------
%             LEMMA 2 TAYLOR EXPANSION OF LL
%---------------------------------

\begin{lem}\label{lem:taylorL}
The extension of $L$ to $\O x^\rho[z]$ has expansion
\[\LL=L_x+L_x'\6_z +{1\over 2!} L''_x\6^2_z+\ldots + {1\over n!}L^{(n)}_x\6_z^n,\]
where the $\C$-linear maps $L_x^{(\ell)}$ act on $\O x^\rho$ while leaving all $z^i$ invariant, and $\6_z$ is the usual differentiation with respect to $z$.
\end{lem}

%-------------------------------------------------
%       LEMMA 3 SUBSTITUTION FOR \zz
%-------------------------------------------------

\begin{lem} \label{lem:substitutionz}
If $L_0$ is an Euler operator of order $n$ with shift $0$, indicial polynomial $\chi_0(s)$, and extension $\LL_0$ to $\O x^\rho[z]$, then

\[\LL_0(x^\rr \zz^i)= x^\rr\cdot[\chi_0(\rr)\zz^i + \chi_0'(\rr) i\zz^{i-1}+{1\over 2!}\chi_0''(\rr) i^{\ul 2}\zz^{i-2}+\ldots + {1\over n!}\chi_0^{(n)}(\rr)i^{\ul n}\zz^{i-n}].\]
\end{lem}
%-------------------------------------------------
%       LEMMA 4 KERNEL LL_0
%-------------------------------------------------

\begin{lem}\label{lem:kernelL0}
The kernel of the extension $\LL_0$ to $\FF_0=\sum_{\rho\in\Omega}\O x^\rho[z]_{<m_\rho}$ of an Euler operator $L_0$ with exponents $\rho\in\Omega\subseteq \C$ of multiplicity $m_\rho$ equals
\[\ds \Ker(\LL_0)= \bigoplus_{\rho\in\Omega}\bigoplus_{i=0}^{m_\rho-1} \C x^\rho z^i.\]
\end{lem}
%-------------------------------------------------
%       LEMMA 5  SOLUTIONS L_0
%-------------------------------------------------

\begin{lem}\label{lem:solutionsL0}
A $\C$-basis of solutions of an Euler equation $L_0y=0$ is given by
\[x^\rho \log(x)^i,\]
where $\rho$ ranges over all local exponents of $L_0$ at $0$ and $0\leq i <m_\rho$, with $m_\rho$ the multiplicity of $\rho$.
\end{lem}

%---------------------------------
%             EXAMPLE TAYLOR EXPANSION
%---------------------------------

\begin{ex}(a) For the Euler operator $L_0=x^2\6^2-3x\6+3$ from before, with indicial polynomial $\chi_0(t)=(t+1)^2$ and exponent $\rho=-1$ of multiplicity $m_\rho=2$,  the extension $\LL_0=x^2\DD^2+3x\DD+1$ to $\O x\inv [z]$ has expansion
\[\LL_0(x^\rho z^i)=x^\rho[(\rho+1)^2z^i + 2(\rho+1)iz^{i-1}+2i(i-1)z^{i-2}]\]
and  kernel
\[\Ker(\LL_0)=\C x\inv \oplus \C x\inv z.\]

%---------------------------------

(b) For the Euler operator $L_0=x^3\6^3 -4x^2\6^2+9x\6-9$ with  indicial polynomial $\chi_0(t)=(t-1)(t-3)^2$ and exponents $1$ and $3$ of multiplicity one and two, respectively, the extension $\LL_0=x^3\DD^3 -4x^2\DD^2+9x\DD-9$ to $\O x[z]\oplus\O x^3[z]$ has expansion
\[\LL_0(x^\rho z^i)= x^\rho[(\rho-1)(\rho-3)^2z^i + (3\rho-5)(\rho-3) iz^{i-1}+(6\rho -14)i^{\ul 2}z^{i-2}+ 6i^{\ul 3}z^{i-3}]\]
and kernel
\[\Ker(\LL_0)=\C x\oplus \C x^3\oplus  \C  x^3 z.\]
\end{ex}
%---------------------------------
%            IMAGE  EULER OPERATORS
%---------------------------------

\begin{rem} In order to apply the perturbation lemma \ref{lem:pert} below to an operator $L$ acting on the space $\FF_0=\sum_{\rho\in\Omega} \O x^\rho[z]_{<m_\rho}$ one has to determine the image of the initial form $\LL_0$ of $\LL$. Write $L=L_0-T$. Assuming that $L_0$ has shift $0$, it follows that $T$ is an operator with shift $>0$, that is, it  increases the order in $x$ of elements of $\FF_0$. Therefore, it sends $\FF_0$ to $\FF_0 x=\sum_{\rho\in\Omega} \O x^{\rho+1}[z]_{<m_\rho}$. One has no control about the precise image of $\TT$: it can be equal to whole $\FF_0 x$ but it can also be much smaller. The perturbation lemma requires in any case the inclusion $\Im(\TT)\subseteq\Im(\LL_0)$ of images. This would trivially hold if $\LL_0$ were surjective onto $\FF_0 x$. But this is not the case in general: it suffices to take $L_0=x^2\6^2-x\6$ with local exponents $\sigma =0$ and $\rho=2$, both of multiplicity one. Then $\FF_0=\O +\O x^2=\O$ and $\LL_0=L_0$. The image of $\FF_0$ under $L_0$ is $L_0(\FF_0)=\C x+\O x^3\subsetneq \O x=\FF_0 x$, with a gap at $x^2$. However, if $L=x^2\6^2-x\6-x=L_0-T$, the operator $T=x$ sends $x\in \FF_0$ to $x^2\not\in L_0(\FF_0)$. So the perturbation lemma does not apply to this situation. The way out of this dilemma is a further enlargement of $\FF_0$ to a carefully chosen function space $\FF$ containing $\FF_0$. This enlargement will be explained in the next section.
\end{rem}

%---------------------------------
%             MAIN THEOREM
%---------------------------------

\subsection{The normal form of differential operators} \label{subsec:nft0}

%---------------------------------
%             REMARK
%---------------------------------

When trying to lift, for an arbitrary operator $L$, the solutions $x^\rho \log x ^k$ of $L_0y=0$ to solutions of $Ly=0$, two obstructions occur. First, $\rho$ might be a multiple root of the indicial polynomial and logarithms already appear in the solutions of $\L_0y=0$. Second, if $\rho$ is not a maximal exponent of $L$ modulo $\Z$, that is, if $\rho+k$ is again an exponent of $L$ for some $k>0$, the lifting poses additional problems since higher powers of logarithms will occur among the solutions. We will approach and solve both problems simultaneously by using the extensions of operators $L$ as defined above to appropriately chosen spaces $\FF$ for which the image of the action of $\LL_0$ on $\FF$ equals $\FF x$. In this situation, the perturbation lemma \ref{lem:pert} will apply to reduce $\LL:\FF\map\FF$ via a linear automorphism of $\FF$ to $\LL_0$.

%---------------------------------
%             ENLARGEMENTS OF FUNCTION SPACES
%---------------------------------

{\bf Enlargement of function spaces.} As was done already classically \cite{Fuchs66} p.~136 and 157, \cite{Fuchs68}, p.~362 and 364, \cite{Thome72}, p.~193, \cite{Frobenius73}, p.~221, it is appropriate to partition the set of exponents of a linear differential operator $L$ into sets $\Omega\subseteq \C$ of exponents whose differences are integers and such that no exponent outside $\Omega$ has integer difference with an element of $\Omega$. We list the elements of each $\Omega$ increasingly,
 \[\rho_1<\rho_2<\cdots<\rho_r,\]
where $\rho_k < \rho_{k+1}$ stands for $\rho_{k+1}-\rho_k\in\N_{>0}$; denote by $m_k\geq 1$ the respective multiplicity of $\rho_k$ as a root of the indicial polynomial $\chi_0$ of $L$ at $0$. Set $n_k=m_1+\cdots+m_k$ and $n_0=0$. To easen the notation, we omit in each $\rho_k$ the reference to the respective set $\Omega=\{\rho_1,\ldots,\rho_r\}$. Instead of $\FF_0^\Omega=\sum_{k=1}^r \O x^{\rho_k}[\zz]_{<m_k}$ we will now allow polynomials in $z$ of larger degree $< n_k$ and take the module 
\[\ds \FF^\Omega=\sum_{k=1}^r \O x^{\rho_k}[\zz]_{<n_k}=\bigoplus_{k=1}^r \bigoplus_{i=n_{k-1}}^{n_k-1}\O x^{\rho_k}z^i
=\bigoplus_{k=1}^{r-1} \bigoplus_{i=0}^{n_k-1}\bigoplus _{\sigma=\rho_k}^{\rho_{k+1}-1}\C x^{\sigma}z^i\oplus \bigoplus_{i=0}^{n_r-1}\O x^{\rho_r}z^i,\]
equipped with the derivation $\DD$ from before (see Figure \ref{fig:image0}). The two different direct sum decompositions of $\FF$ will become relevant in a moment. Then set
\[\ds \FF=\bigoplus_\Omega \FF^\Omega,\]
the sum varying over all sets $\Omega$ of exponents with integer difference. As each summand $ \bigoplus_{i=n_{k-1}}^{n_k-1}\O x^{\rho_k}z^i$ of $\FF^\Omega$ has rank $m_k$, it follows that $\FF$ is free of rank $n$ over $\O$.

%---------------------------------
%             FIGURE 1    FUNCTION SPACE
%---------------------------------

%\cl{\functionspace}
\begin{ex}
We illustrate the construction of the space $\mathcal{F}^\Omega$ with an example. Let
\[L=x^5\partial^5-2x^4\partial^4-2x^3\partial^3+16x^2\partial^2-16x\partial-x.\]
It has indicial polynomial $\chi(s)=s^2(s-2)(s-5)^2$. Therefore the local exponents are given by $\rho_1=0$, $\rho_2=2$ and $\rho_3=5$ with multiplicities $m_1=2$, $m_2=1$ and $m_3=2$.  All local exponents differ by integers and we set $\Omega=\{0,2,5\}$ as well as $n_1=2$, $n_2=3$ and $n_5=5$. Then the space $\mathcal{F}^\Omega$ is given by
\[\mathcal{F}^\Omega=\mathcal{O} \oplus  \mathcal{O}z \oplus \mathcal{O}x^2z^2 \oplus \mathcal{O}x^5z^3 \oplus \mathcal{O}x^5z^4.\]
The exponents $(k,i)$ of  monomials $x^k z^i$ in $\FF^\Omega$ are depicted in Figure~\ref{fig:image0}.
\end{ex}

\begin{figure}[htb]
\centering
%\begin{asy}
%settings.outformat = "pdf";
%size(7cm);
%Label k = Label("$k$", position=EndPoint);
%Label i = Label("$i$", position=EndPoint);
%draw((0,0) -- (9,0), arrow=Arrow(TeXHead), L=k);
%draw((0,0) -- (0,5), arrow=Arrow(TeXHead), L=i);
%real r = 0.3;
%for(int j=1; j < 8; ++j){
%  fill(circle((j,0),r), mediumgray);
%  fill(circle((j,1),r), mediumgray);
%};
%for(int j=2; j < 8; ++j){
%  fill(circle((j,2),r), mediumgray);
%};
%for(int j=5; j < 8; ++j){
%  fill(circle((j,3),r), mediumgray);
%  fill(circle((j,4),r), mediumgray);
%};
%fill(circle((0,0),r), heavyred);
%fill(circle((0,1),r), heavyred);
%fill(circle((2,0),r), heavyred);
%fill(circle((5,0),r), heavyred);
%fill(circle((5,1),r), heavyred);
%label("$\cdots$", (8.5,2));
%label("$\rho_1$", (0,-1));
%label("$\rho_2$", (2,-1));
%label("$\rho_3$", (5,-1));
%label("$n_1-1$", (-2,1));
%label("$n_2-1$", (-2,2));
%label("$n_3-1$", (-2,4));
%
%\end{asy}
\includegraphics{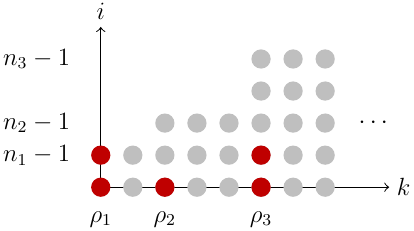}
\caption{The sets of exponents $(k,i)$ of  monomials $x^k z^i$ in $\FF^\Omega$; in red monomials in $\Ker(\LL_0)$.}
\label{fig:image0}
\end{figure}

%---------------------------------
%             EXAMPLE FUNCTION SPACE
%---------------------------------

\begin{ex} 
This example will illustrate why local exponents with integer difference have to be treated in a separate and quite peculiar way. Assume that the Euler operator $L_0$ has just two local exponents $\sigma$ and $\rho$ of multiplicities $m_\sigma$ and $m_\rho$, respectively, say $\Omega =\{\sigma,\rho\}$. If $\rho-\sigma\not\in \Z$, then
\[\FF= \O x^\sigma[z]_{<m_\sigma}\oplus \O x^\rho[z]_{<m_\rho};\]
if $\rho-\sigma\in \N$, then
\[\FF= \O x^\sigma[z]_{<m_\sigma}+ \O x^\rho[z]_{<m_\sigma+m_\rho}= \O x^\sigma[z]_{<m_\sigma}\oplus 
\O x^\rho z^{m_\sigma}[z]_{<m_\rho}.\]

The extension $\LL_0$ of $L_0$ to $\FF$ has kernel $\C x^\sigma[z]_{<m_\sigma}\oplus \C x^\rho[z]_{<m_\rho}$ in the first case, and $\C x^\sigma[z]_{<m_\sigma}\oplus\C x^\rho z^{m_\sigma}[z]_{<m_\rho}$ in the second case. The respective images of $\LL_0$ are
\[\O x^{\sigma+1}[z]_{<m_\sigma}\oplus \O x^{\rho+1}[z]_{<m_\rho}\]
and
\[\O x^{\sigma+1}[z]_{<m_\sigma}\oplus
 \O x^{\rho+1}z^{m_\sigma}[z]_{<m_\rho},\]

so they equal $\FF x$ in both cases.

%---------------------------------
%             NEW TEXT HH FROM DEC 21 2022
%---------------------------------

If we would have taken in the second case where $\rho - \sigma\in\N$ is integral the space

\[\FF=\O x^\sigma[z]_{<m_\sigma} + \O x^\rho[z]_{<m_\rho},\]

the image of $\LL_0$ would have been

\[\bigoplus_{k=1}^{\rho-\sigma-1}\C x^{\sigma+k}[z]_{<m_\sigma} \oplus \C x^\rho [z]_{<m_\sigma-m_\rho}\oplus \O x^{\rho+1}[z]_{<\max(m_\rho, m_\sigma)} \subsetneq \FF x,\]

which is strictly included in $\FF x$. Here $\C x^\rho [z]_{<m_\sigma-m_\rho}$ is to be read as $0$ for $m_\sigma\leq m_\rho$. Indeed, $x^\rho z^{m_\sigma-1}\in \FF x$ is not in the image of $\LL_0$. This would cause serious obstructions when trying to see $\LL$ on $\FF$ as a (negligible) perturbation of $\LL_0$, since the higher order terms of $\LL$ may produce images in whole $\FF x$. So the Perturbation Lemma \ref {lem:pert} below would not apply.
\end{ex}

The example suggests to admit in $\FF$ powers of the logarithm, say, of $z$, which exceed the respective multiplicity of the local exponent $\rho$ appearing in the factor $x^\rho$. The following lemma gives a precise answer of how to proceed; see Lemma~\ref{lem:surjectivity} for the corresponding result in characteristic $p$.

%---------------------------------
%    LEMMA IMAGE EULER OPERATORS
%---------------------------------

\begin{lem}\label{lem:image}
Let $L\in \O[\6]$ be an Euler operator with shift $\tau=0$. Denote by $\Omega=\{\rho_1,\ldots,\rho_r\}$ a set of increasingly ordered local exponents $\rho_k$ of $L$ with integer differences and multiplicities $m_k$. Set $n_k=m_1+\ldots +m_k$ and $\FF=\FF^\Omega =\sum_{k=1}^r \O x^{\rho_k}[\zz]_{<n_k}$.

(a) The induced map $\LL:\mathcal{F}\to \mathcal{F}$ has image $\ds \Im(\LL)=\FF x=\sum_{k=1}^r \O x^{\rho_k+1}[\zz]_{<n_k}$. 

(b) Its kernel $\Ker(\LL)=\bigoplus_{k=1}^r \C x^{\rho_k}[\zz]_{<m_k}$ (cf.~Lemma~\ref{lem:kernelL0}) has direct complement  
\[\ds \HH= \bigoplus_{k=2}^r \bigoplus_{i=m_k}^{n_k-1}\C x^{\rho_k} z^i
\oplus
\bigoplus_{k=1}^{r-1}\bigoplus _{e=1}^{\rho_{k+1}-\rho_k-1} \bigoplus_{i=0}^{n_k-1}\C x^{\rho_k+e}z^i\oplus \bigoplus_{i=0}^{n_r-1}\O x^{\rho_r+1}z^i,\]
in $\FF$. Thus the restriction ${\LL}_{\vert\HH}$ defines a linear isomorphism between $\HH$ and $\FF x$.

\end{lem}

%---------------------------------
%             PROOF LEMMA IMAGE
%---------------------------------

\begin{proof}
(a) We show first that $\LL$ sends $\FF$ into $\FF x$. Recall from Lemma \ref{lem:substitutionz} that 
\[\LL(x^\rr \zz^i)= x^\rr\cdot[\chi(\rr)\zz^i + \chi'(\rr) i\zz^{i-1}+{1\over 2!}\chi''(\rr) i^{\ul 2}\zz^{i-2}+\ldots + {1\over n!}\chi^{(n)}(\rr)i^{\ul n}\zz^{i-n}].\]

Therefore, as $\chi^{(\ell)}(\rho_k)=0$ for $0\leq \ell<m_k$, and using that $n_k-m_k=n_{k-1}$ for $k\geq 2$, it follows that $\LL$ sends $\FF$ into 
\[\ds\sum_{k=1}^r \O x^{\rho_k}[\zz]_{<n_k-m_k}=
\sum_{k=2}^r \O x^{\rho_k}[\zz]_{<n_{k-1}}\subseteq 
\sum_{k=2}^r \O x^{\rho_{k-1}+1}[\zz]_{<n_{k-1}}\subseteq \FF x.\]
Here, we use that $\rho_k-\rho_{k-1}\geq 1$. This proves $\LL(\FF)\subseteq \FF x$.

For the inverse inclusion $\LL(\FF)\supseteq\FF x$ it suffices to check that all monomials $x^\sigma z^i\in \FF x$ lie in the image, where $\sigma=\rho_k+e$ for some $k=1,\ldots,r$ and $e\geq 1$, and where $i<n_k$. %This is actually the trickiest part of the proof. 
We distinguish two cases.

(i) If $\sigma\not\in\Omega$, proceed by induction on $i$. Let $i=0$. By Lemma \ref{lem:taylorL},
\[\ds \LL(x^\sigma) = L_x(x^\sigma)+\sum_{j=1}^n {1\over j!}L_x^{(j)}\6_z^j(x^\sigma) = L_x(x^\sigma) = \chi(\sigma) x^\sigma \neq 0,\]
since $\sigma$ is not a root of $\chi$. So $x^\sigma\in \LL(\FF)$. Let now $i>0$. Lemmata \ref{lem:taylorL} and \ref{lem:substitutionz} yield
\[\ds \LL(x^\sigma z^i) = L_x(x^\sigma z^i)+\sum_{j=1}^n {1\over j!}L_x^{(j)}\6_z^j(x^\sigma z^i) = \chi(\sigma) x^\sigma z^i+\chi^{(j)}(\sigma)x^\sigma\sum_{j=1}^n {i^{\ul j}\over j!} z^{i-j}.\]

By the inductive hypothesis and using again that $\chi(\sigma)\neq 0$ we end up with $x^\sigma z^i\in \LL(\FF)$.

%---------------------------------

(ii) If $\sigma\in\Omega$, write $\sigma =\rho_k$ for some $1\leq k\leq r$. As $x^\sigma z^i=x^{\rho_k}z^i\in \FF x$ and $\rho_1<\rho_2<\cdots <\rho_r$, we know that $k\geq 2$ and
\[\ds x^{\rho_k} z^i\not \in x\cdot \sum_{\ell=k}^r \O x^{\rho_\ell}[\zz]_{<n_\ell}.\]
Hence 
\[\ds x^{\rho_k} z^i\in x\cdot \sum_{\ell=1}^{k-1}\O x^{\rho_\ell}[\zz]_{<n_\ell}.\]
This implies in particular that $0\leq i<n_{k-1}$, which will be used later on. We proceed by induction on $i$. Let $i=0$. By Lemma \ref{lem:taylorL},
\begin{align*}
\ds \LL(x^{\rho_k}z^{m_k}) &=\sum_{j=0}^{m_k-1} {1\over j!}L_x^{(j)}\6_z^j(x^{\rho_k}z^{m_k}) + {1\over m_k!}L_x^{(m_k)}\6_z^{m_k}(x^{\rho_k}z^{m_k}) +
\sum_{j=m_k+1}^n {1\over j!}L_x^{(j)}\6_z^j(x^{\rho_k}z^{m_k})\\
&=\sum_{j=0}^{m_k-1} {(m_k)^{\ul j}\over j!}\chi^{(j)}(\rho_k)x^{\rho_k}z^{m_k-j} + \chi^{(m_k)}(\rho_k)x^{\rho_k}\\
&=\chi^{(m_k)}(\rho_k)x^{\rho_k}.
\end{align*}
Here, the sum in the first summand in the last but one line is $0$ since $\rho_k$ is a root of $\chi$ of multiplicity $m_k$, and for the same reason, the second summand $\chi^{(m_k)}(\rho_k)x^{\rho_k}$ is non-zero. So $x^\sigma = x^{\rho_k}\in \LL(\FF)$. Let now $i>0$ and consider $x^\sigma z^i=x^{\rho_k}z^i\in\FF x$. We will use that $i < n_{k-1}$ as observed above. Namely, this implies that $m_k+i< m_k+n_{k-1}=n_k$, so that $x^{\rho_k}z^{m_k+i}$ is an element of $\FF$. Let us apply $\LL$ to it. Similarly as in the case $i=0$ we get
\begin{align*}
\ds \LL(x^{\rho_k}z^{m_k+i}) &= \sum_{j=0}^{m_k-1} {1\over j!}L_x^{(j)}\6_z^j(x^{\rho_k}z^{m_k+i}) +{1\over m_k!}L_x^{(m_k)}\6_z^{m_k}(x^{\rho_k}z^{m_k+i}) +{}\\
\ds {}&\quad +
\sum_{j=m_k+1}^n {1\over j!}L_x^{(j)}\6_z^j(x^{\rho_k}z^{m_k+i})\\
\ds {}&={(m_k+i)^{\ul m_k}\over m_k!}\chi^{(m_k)}(\rho_k)x^{\rho_k}z^i +\sum_{j=m_k+1}^n {(m_k+i)^{\ul j}\over j!}\chi^{(j)}(\rho_k)x^{\rho_k}z^{m_k+i-j}.
\end{align*}

The sum appearing in the second summand of the last line belongs to $\LL(\FF)$ by the induction hypothesis since $m_k+i-j<i$.  As $\chi^{(m_k)}(\rho_k)\neq 0$, we end up with $x^\sigma z^i=x^{\rho_k}z^i\in \LL(\FF)$. This proves the inverse inclusion $\LL(\FF)\supseteq\FF x$ and hence assertion (a).

(b) From the shape of $\FF$ and $\Ker(\LL)$ as depicted in Figure~\ref{fig:image0} one sees that $\HH$ is a direct complement of $\Ker(\LL)$ in $\FF$. Hence ${\LL}_{\vert\HH}$ is automatically injective and $\LL(\FF)=\LL(\HH)=\FF x$. 
\end{proof}

Here is the result from functional analysis required for the proof of the normal form theorems in characteristic $0$ and $p$.

%---------------------------------
%           END   NEW TEXT HH FROM DEC 21 2022
%---------------------------------

\begin{lem}[Perturbation Lemma]\label{lem:pert}
If $\ell:F\map G$ is a continuous linear map between complete metric vector spaces which decomposes into $\ell=\ell_0-t$ with $\Im(t)\subseteq\Im(\ell_0)$ and satisfies $\abs{s(t(f))}\leq C\cdot \abs f$, $0<C<1$, for a right inverse $s:\Im(\ell_0)\map F$ of $\ell_0:F\map \Im(\ell_0)$ and all $f\in F$, then $u=\Id_F-st$ is a continuous linear automorphism of $F$ which transforms $\ell$ into $\ell_0$ via $\ell u\inv =\ell_0$.
\end{lem}

%---------------------------------
\begin{proof}
The prospective inverse of $u$ is $v=\sum_{k=0}^\infty (st)^k$. It is well defined and continuous because of the estimate for $st(f)$ and the completeness of $F$. Hence $u$ is an automorphism of $F$. From $\ell_0s=\Id_{\Im(\ell_0)}$ it follows that $\ell_0 s\ell_0=\ell_0$. From $\Im(t)\subseteq\Im(\ell_0)$ one gets that the compositions $st$ and $s\ell$ are well defined and that $\ell_0s\ell=\ell$ holds. Then 
\begin{align*}
\ell_0 u &= \ell_0(\Id_F-s t)\\
&= \ell_0(\Id_F-s (\ell_0-\ell))\\
&=\ell_0(\Id_F-s \ell_0+s \ell)\\
&= \ell_0-\ell_0 s \ell_0+\ell_0 s \ell\\
&= \ell_0 s\ell\\
&=\ell,
\end{align*}
as required. This proves the result.
\end{proof}  
 
%---------------------------------
%         NORMAL FORM THEOREM 
%---------------------------------

\begin{thm}[Normal form theorem in characteristic $0$]\label{thm:nft0}
Let $L\in \O[\6]$ be a linear differential operator with holomorphic coefficients at $0$, initial form $L_0$ and shift $\tau=0$. Denote by $\Omega=\{\rho_1,\ldots,\rho_r\}$ a set of increasingly ordered local exponents $\rho_k$ of $L$ with integer differences and multiplicities $m_k$. Set $n_k=m_1+\ldots +m_k$ and $\ds \FF=\FF^\Omega =\sum_{k=1}^r \O x^{\rho_k}[\zz]_{<n_k}$. Let $\LL,\LL_0$ act on $\FF$ via $\DD x=1$ and $\DD z=x\inv$ as above. Assume that $L$ has a regular singularity at $0$.

%---------------------------------

(a)  The composition of the inverse $({\LL_0}_{\vert \HH})\inv:\FF x\map \HH$ of ${\LL_0}_{\vert \HH}$ with the inclusion $\HH\subseteq\FF$ defines a right inverse $\SS_0:\FF x\map \FF $ of $\LL_0$, again denoted by $({\LL_0}_{\vert \HH})\inv$. Let $\TT: \FF\map \FF x$ be the extension of $T=L_0-L$ to $\FF$. The map \[u=\Id_\FF - \SS_0\circ \TT:\FF\map \FF\]
is a linear automorphism of $\FF$, with inverse $\ds v=u\inv=\sum_{k=0}^\infty (\SS\circ\TT)^k:\FF\map \FF$.

(b) The automorphism $v$ of $\FF$ transforms $\LL$ into $\LL_0$,\[\LL\circ v=\LL_0.\]

(c) If $0$ is an arbitrary (i.e., regular or irregular) singularity of $L$, statements (a) and (b) hold true with $\O$ replaced by the ring $\wh \O$ of formal power series over $\C$ or over an arbitrary algebraically closed field $K$ of characteristic $0$.
\end{thm}

%---------------------------------
%    IGNORE  : SHAPE OF SS_0
%---------------------------------
\ignore

(b) [...adjust, including complements of kernels] A right inverse of $\LL_0$ can be given by an operator $\SS_0: \FF x\map \FF$ of the form \med

\cl{$\ds \SS_0(x^\sigma z^i)= 
{1\over \chi_0(\sigma)}
\left[x^\rho z^i - {\chi_0'(\sigma)i\over \chi_0(\sigma)}x^\rho z^{i-1} - \left[{1\over 2}{\chi_0''(\sigma)i(i-1)\over \chi_0(\sigma)}-{\chi_0'(\sigma)^2i\over \chi_0(\sigma)^2}\right]x^\rho z^{i-2}- \ldots \right]$.}\med

\recognize
%---------------------------------
%       PROOF OF THEOREM
%---------------------------------
\ignore

{\bf Proof.} As $\FF=\bigoplus \FF^\Omega$ and $\LL$, $\LL_0$ and $\TT$ are direct sums of the respective restrictions $\FF^\Omega\map \FF^\Omega$, it suffices to show the assertion for each single summand $\FF^\Omega$. So fix one set $\Omega$ of exponents with integer differences, and write $\FF$ instead of $\FF^\Omega$ for convenience. In particular, our slight abuse of notation for the sets $\Omega$ in the statement of the theorem is thus avoided.\med

\recognize
%---------------------------------
%          PROOF    PART (a)
%---------------------------------

\begin{proof}
%---------------------------------
%          PROOF    PART (c) + (d)
%---------------------------------
Note first that $u$ is well defined since the map $T$ increases the order of series and thus sends $\FF$ into $\FF x$.

Once we show that $\abs{\SS_0(\TT(f))}\leq C\abs f$ holds for some $0<C<1$ and all $f\in \FF$, the perturbation lemma \ref{lem:pert} implies that $u=\Id_\FF-\SS_0\circ \TT$ is a linear automorphism of $\FF$ with $\LL\circ u\inv = \LL_0$, proving assertions (a) to (c) of the theorem. The proof of the estimate is split into two parts, first for formal power series and then for convergent ones, and uses a different metric in each case.

%---------------------------------

(i) Formal case: Denote by $\wh \O =K\osb x\csb$ the formal power series ring over an arbitrary field $K$ of characteristic $0$, equipped with the metric $d(f,g)=2^{-\ord_0(f-g)}$, where $\ord_0$ denotes the order of vanishing at $0$. Let $\wh \FF$ denote the induced $\wh \O$-modules $\wh \FF=\FF\otimes_K \wh\O$ and write again $\LL$ for the extension $\wh \LL$ to $\wh \FF$. As $\TT$ increases the order of series in $\wh \O$, while $\LL_0$ and $\SS_0$ do not decrease the order, it follows that also $\SS_0\circ \TT$ increases the order. It thus satisfies the inequality $\abs{\SS_0(\TT(f))} \leq C\cdot \abs f$ from the beginning, for some $0<C<1$, having set $\abs f= d(f,0)=2^{-\ord f}$. Therefore the von Neumann series
\[v=\sum_{j=0}^\infty (\SS_0\circ \TT)^j\]
defines a $\C$-linear map $v:\wh \FF\map \wh \FF$. Then it is clear that $v=u\inv=(\Id_{\wh \FF}-\SS\circ \TT)\inv$. So $u$ and $v$ are automorphisms, and $\LL\circ v =\LL_0$ by the perturbation lemma. This proves assertion (c) as well as the formal version of (a).

%---------------------------------

(ii) Convergent case: To prove the same thing inside $\O$, denote by $\O_s$ the subring of germs of holomorphic functions $h$ such that $\abs h_s <\infty$. Here, $s>0$ and $\abs{\sum_{k=0}^\infty a_kx^k}_s:=\sum_{k=0}^\infty \abs {a_k}s^k$. It is well known that the rings $\O_s$ are Banach spaces, and that $\O=\bigcup_{s>0} \O_s$ \cite{GR71}. For $s>0$ sufficiently small, $u$ restricts to a linear map $u_s$ on the induced Banach space $\FF_s=\left(\sum_{k=1}^r \O x^{\rho_k}[\zz]_{<n_k}\right)_s$. For the convergence of $v_s$ it therefore suffices to prove that $\abs{\abs{\SS_0\circ \TT}}_s<1$, where $\abs {\abs-}_s$ denotes the operator norm of bounded linear maps $\FF_s\map \FF_s$. Once this is proven, $v_s=u_s\inv$ holds as before and shows that $u_s$ and hence also $u$ are linear isomorphisms. This argument provides a compact reformulation of Frobenius' proof for the convergence of solutions \cite{Frobenius73}, p.~218.

The inequality $\abs{\abs{\SS_0\circ \TT}}_s<1$ is equivalent to the existence of a constant $0<C<1$ such that
\[\abs{\SS_0(\TT(x^\rr h(x)z^i))}_s \leq C\cdot \abs {x^\rr h(x)z^i}_s\]
for all $x^\rr h(x)z^i\in \FF_s$. This will imply in particular that $(\SS_0\circ \TT)(\FF_s)\subseteq \FF_s$. 

%---------------------------------

We will treat the case where $\rho$ is a maximal local exponent of $L$ modulo $\Z$ and a simple root of $\chi_0$. In this case, no extensions of operators are required, and we can work directly with $L$, $S$ and $T$ and $\FF=\O x^\rho$. For non-maximal exponents there occur notational complications which present, however, no substantially new difficulty. So we shall omit the general case. For $h=\sum_{k=0}^\infty a_kx^k\in \O$ and writing $L=\sum_{j=0}^n p_j(x)\6^j$ with $p_j=\sum_{i=0}^\infty c_{ij}x^i$ we have 
\[\ds T(x^\rr h) = -\sum_{i-j>0}\sum_{k=0}^\infty (\rr+k)^{\ul j}\, c_{ij}a_kx^{\rr+k+i-j},\]
and, recalling that $L_0$ is assumed to have shift $0$,
\[\ds S(T(x^\rr h))= -\sum_{i-j>0}\sum_{k=0}^\infty {(\rr+k)^{\ul j}\over \chi_L(\rr+k+i-j)}\, c_{ij}a_kx^{\rr+k+i-j}.\]
As $i-j>0$, $k\geq 0$, and $\rr$ is maximal, no $\rr+k+i-j$ appearing in the denominator is a root of $\chi_L$. Hence the ratio 
\[\ds {(\rr+k)^{\ul j}\over \chi_L(\rr+k+i-j)}=   {(\rr+k)^{\ul j}\over \sum_{\ell=0}^n c_{\ell,\ell}(\rr+k+i-j)^{\ul \ell}}\]
is well defined. But $c_{n,n}\neq 0$ since $0$ is a regular singularity of $L$, and hence $(\rr+k+i-j)^{\ul n}$ appears in the denominator with non-zero coefficient. As $j\leq n$ this ensures that the ratio remains bounded in absolute value, say $\leq c$, as $k$ tends to $\infty$. Taking norms on both sides of the above expression for $S(T(x^\rr h))$ yields, for $s\leq 1$ and $h\in \O_s$, the estimate
\[\ds \abs{S(T(x^\rr h))}_s\leq 
c\sum_{i-j>0}\sum_{k=0}^\infty \abs{c_{ij}}\abs{a_k}s^{\rr+k+i-j}= 
c\sum_{i-j>0} \abs{c_{ij}}s^{i-j}\sum_{k=0}^\infty \abs{a_k}s^{\rr+k}\]

But by assumption, $p_j=\sum_{i=0}^\infty c_{ij}x^i\in \O_s$ for all $0<s\leq s_0$ and all $j=0,\ldots,n$. This implies in particular $\sum_{i>j}^\infty c_{ij}x^i\in \O_s$ and then, after division by $x^{j+1}$ and since $i> j$, that
\[\ds \sum_{i>j}^\infty c_{ij}x^{i-(j+1)}\in \O_s.\]
We get for all $s\leq r$ that 
\[\ds\sum_{i-j>0} \abs{c_{ij}}s^{i-j}=s\cdot \sum_{i-j>0} \abs{c_{ij}}s^{i-(j+1)}\leq c' s\]
for some $c'>0$ independent of $s$. This inequality allows us to bound $\abs{S(T(x^\rr h))}_s$ from above by
\[\ds \abs{S(T(x^\rr h))}_s\leq 
cc's\sum_{k=0}^\infty \abs{a_k}s^{\rr+k}=cc's\abs {x^\rr h}_s.\]
Take now $s_0>0$ sufficiently small, say $s_0\leq \min(1,r)$ and $s_0< {1\over cc'}$, and get a constant $0<C<1$ such that for $0<s\leq s_0$ one has
\[\abs {S(T(x^\rr h))}_s\leq C\cdot \abs {x^\rr h}_s.\]
This establishes $\abs{\abs{S\circ T}}_s<1$ on $\FF_s$ for $0<s\leq s_0$. By the Perturbation Lemma \ref{lem:pert}, $u_s=\Id_{\FF_s}-S\circ T$ is an automorphism of $\FF_s$ with inverse $v_s=\sum_k(S\circ T)^k$. This completes the proof of the theorem.\end{proof}
%---------------------------------
%             APPLICATIONS
%---------------------------------

\subsection{Solutions of regular singular equations} \label{subsec:sol0}

As a first consequence of the normal form theorem \ref{thm:nft0} we recover the classical theorem of Fuchs from 1866 and 1868 about the local solutions of differential equations at regular singular points \cite{Fuchs66}, \cite{Fuchs68}. The statement was reorganized and further detailed by Thom\'e and Frobenius in a series of papers between 1872 and 1875 \cite{Thome72}, \cite{Thome73}, \cite{Thome73b}, \cite{Frobenius73}, \cite{Frobenius75}. See also \cite{Fabry85} formula (9), p.~19.

%---------------------------------
%     THM SOLUTIONS  FUCHS FROBENIUS 
%---------------------------------

\begin{thm}[Local solutions in characteristic 0]\label{thm:sol0}  Let $L\in \O[\6]$ be a linear differential operator with holomorphic coefficients and regular singularity at $0$. For each set $\Omega$ of local exponents of $L$ with integer differences, let $u_\Omega:\FF^\Omega\map \FF^\Omega$ be the automorphism of assertion (a) of the normal form theorem. 

(a) Varying $\Omega$, a $\C$-basis of local solutions of $L y=0$ at $0$ is given by
\[y_{\rho,i}(x)=u_\Omega \inv(x^\rho z^i)_{\vert z=\log x },\]
for $\rho\in \Omega$ a local exponent of $L$ of multiplicity $m_\rho$, and $0\leq i< m_\rho$. 

(b) Order the exponents in a chosen set $\Omega$ as $\rho_1<\ldots <\rho_r$ and write $m_k$ for $m_{\rho_k}$. Set $n_k=m_1+\ldots + m_k$. Each solution related to $\Omega$ is of the form, for $1\leq k\leq r$ and $0\leq i<m_k$,

\[\ds y_{\rho_k,i}(x)=x^{\rho_k}[f_{k,i}+\ldots +f_{k,0}\log(x)^i]+\sum_{\ell=k+1}^r x^{\rho_\ell} \sum_{j=n_{\ell-1}}^{n_\ell-1} h_{k,i,j}(x) \log(x)^j,\]

with holomorphic $f_{k,i}$ and $h_{k,i,j}$ in $\O$, where $f_{k,0}$ has non-zero constant term.
\end{thm}
%---------------------------------
%             PROOF
%---------------------------------

\begin{proof}
Let $\Omega$ be a set of local exponents of $L$ at $0$ with integer differences and consider the space $\FF^\Omega =\sum_{k=1}^r \O x^{\rho_k}[\zz]_{<n_k}$ as in the statement of the normal form theorem. Extend $L$ and $L_0$ to $\FF=\bigoplus_\Omega\FF^\Omega$. By Lemma \ref{lem:kernelL0}, a $\C$-basis of solutions of $\LL_0$ is given by the monomials $x^\rho z^i$, $0\leq i \leq m_\rho-1$, where $\rho$ is a local exponent of multiplicity $m_\rho$. By assertion (d) of the normal form theorem and since $L$ and $L_0$ have the same order $n$, the pull-backs $u\inv(x^\rho z^i)$ form a $\C$-basis of solutions of $\LL y=0$. Now Lemma \ref{lem:substitutionlog} gives the result.
\end{proof}

%---------------------------------
%             REMARK   COEFFICIENT FUNCTIONS OF SOLUTIONS
%---------------------------------

\begin{rem} The coefficient functions $f_{k,i}$ and $h_{\rho,i,j}\in\O$ of the solutions in assertion (b) of the theorem are related to each other. For instance, if $\rho$ is a maximal exponent in $\Omega$ of multiplicity $m_\rho$, then 
\begin{align*}
y_{\rho,0}&=x^\rr\cdot g_0\\
y_{\rho,1}&=x^\rr\cdot [g_1+g_0 \log x ]\\
&\ldots\\
y_{\rho,m_\rho-1}&=x^\rr\cdot [g_{m_\rho-1}+ g_{m_\rho-2}\log x +\ldots + g_1 \log(x)^{m_\rho-2} + g_0 \log(x)^{m_\rho-1}], 
\end{align*}
with holomorphic $g_0,\ldots,g_{m_\rho-1}$, where $g_0$ has non-zero constant term.
\end{rem}
%---------------------------------
%          EXAMPLE    SOLUTIONS TWO EXPONENTS
%---------------------------------

\begin{ex} (i) If $L$ has exactly two exponents $\sigma$ and $\rho$, with $\rho-\sigma\in\N_{>0}$ and of multiplicities $m_\sigma$ and $m_\rho$, respectively, we get accordingly
\[\FF= x^\sigma[\O \oplus \cdots \oplus \O z^{m_\sigma -1}] +  x^\rho [\O\oplus \cdots\oplus \O z^{m_\sigma+m_\rho-1}].\]
which we rewrite as
\[\FF= x^\sigma[\O \oplus \cdots \oplus \O z^{m_\sigma -1}] \oplus  x^\rho [\O z^{m_\sigma}\oplus \cdots\oplus \O z^{m_\sigma+m_\rho-1}].\]
A basis of solutions of $Ly=0$ are $\O$-linear combinations
\begin{align*}
y_{\sigma,0}&=x^\sigma\cdot h_0+x^\rho g_0\log x ^{m_\sigma}\\
y_{\sigma,1}&=x^\sigma\cdot [h_1+h_0 \log x ]+x^\rho\log(x)^{m_\sigma}[g_1+g_0\log x ]\\
&\ldots\\
y_{\sigma,m_\sigma -1}&=x^\sigma\cdot [h_{m_\sigma-1}+ h_{m_\sigma-2}\log x +\cdots + h_1 \log(x)^{m_\sigma-2} + h_0 \log(x)^{m_\sigma-1}]+ {}\\
&\quad{}+x^\rho\log(x)^{m_\sigma}\cdot[g_{m_\rho-1}+\cdots+ g_0 \log(x)^{m_\rho -1}]\\
y_{\rho,0}&=x^\rr\cdot f_0\\
y_{\rho,1}&=x^\rr\cdot [f_1+f_0 \log x ]\\
&\ldots\\
y_{\rho,m_\rho-1}&=x^\rr\cdot [f_{m_\rho-1}+ f_{m_\rho-2}\log x +\ldots + f_1 \log(x)^{m_\rho-2} + f_0 \log(x)^{m_\rho-1}]
\end{align*}
with holomorphic $f_0,\ldots,f_{m_\rho-1}$, $g_0,\ldots,g_{m_\rho-1}$, $h_0,\ldots,h_{m_\sigma-1}$.\\

(ii) The function $e^x\log x $ satisfies the differential equation $Ly=0$ for 
\[L=x^2\partial^2+(1-2x)x\partial + x(x-1).\]
A basis of solutions is completed by $e^x$. The initial form of $L$ is 
\[L_0= x^2\partial^2+x\partial.\]
Consequently, the only local exponent of $L$ is $0$ with multiplicity $2$. The basis of solution is, as expected, contained in
\[\mathcal{O} \oplus \mathcal{O}z.\]
\end{ex}

\subsection{Applications in characteristic zero}\label{subsec:app0}
%---------------------------------
%             IRREGULAR SINGULARITIES
%---------------------------------
{\bf Irregular singularities.} Whenever the point $0$ is an irregular singularity of a differential operator $L\in\O[\6]$ with holomorphic coefficients, i.e., when $n_0=\ord\,L_0<\ord\, L=n$, Theorem \ref{thm:sol0} does not provide a basis of solutions of $Ly=0$, but only $n_0$ linearly independent solutions thereof. It is well known that the solutions which are missing for a full basis are more complicated and may have essential singularities \cite{Fabry85}. More specifically, they are of the form
\[y(x)=\exp(q(x)) \cdot x^\rho \cdot \left[h_0(x)+h_1(x)\log x +\ldots+h_k(x)\log(x)^k\right],\]
where $q\in\C(x)$ is a rational function, $\rho\in\C$ a local exponent of $L$, and $h_i$ holomorphic \cite{Salvy19}, Thm.~3, \cite{Ince44}, Chap.~XVII, p.~417.  Actually, one can even take for $q$ a Laurent polynomial
\[ q(x) =\sum_{r\in\Q_{>0}} c_r x^{-r},\]
with $c_r\in\C$, almost all $c_r=0$. It suffices to take here $r>0$ since summands with non-negative exponents produce holomorphic factors in $y(x)$. In \cite{Merkl22}, Nicholas Merkl describes an algorithm how to construct these solutions by reducing the differential equation $Ly=0$ to various differential equations $\wt Ly=0$, all with regular singularity at $0$, to apply then to these new equations the normal form theorem in characteristic $0$, Theorem \ref{thm:nft0}, to obtain their respective solutions as in \ref{thm:sol0}. It then suffices to pull back these solutions to the original equation via the inverse conjugations. Doing this for all induced equations $\wt Ly=0$, one eventually obtains a basis of solutions of $Ly=0$.

This shows that the normal form theorem \ref{thm:nft0} is applicable to {\it all} linear differential equations with holomorphic coefficients to construct their solutions. In the irregular case, there is also a method to find the solutions using the {\it Newton polygon} of $L$: it is similar in substance, though more computational, see \cite{SvdP03}, section~3.3, p.~90.
\med

We briefly sketch Merkl's algorithm: Let be given an operator $L\in\O[\6]$ of order $n$. Denote by $\delta=x\6$ the basic Euler operator, and define, for $r\in \Q_{\geq 0}$ a positive rational number, the weighted operator
\[\delta_r= x^r\delta.\]
Here, $x^r$ is considered either as a symbol or as a {\it Puiseux monomial} with $(x^r)'=rx^{r-1}$. Writing $r=e/d$ with $e,d\in\N$, we may then expand formally $L$ as a linear combination
\[ L=\sum_{j=0}^n \sum_{i=0}^\infty c_{ij} x^{i/d}\delta_r^j,\]
with coefficients $c_{ij}x^{i/d}$. For each $j$, let $i=i_j\in \N$ be minimal with $c_{ij}\neq 0$ (we suppress here the reference to $r$). Then define the {\it weighted initial form} $L_{0,r}$ and the {\it weighted indicial polynomial} $\chi_r$ of $L$ with respect to $r$ as
\begin{align*}
L_{0,r}&=\sum_{j=0}^n c_{i_jj} \delta_r^j\in \C[\delta_r],\\
\chi_r&=\sum_{j=0}^n c_{i_jj} s^j\in \C[s].
\end{align*}

For $r=0$ we just get the classical initial form $L_0=L_{0,0}$ and its indicial polynomial $\chi=\chi_0$ defined earlier. Note that for generic $r$, the polynomial $\chi_r$ will be a monomial and hence have the unique root $0$ in $\C$. The interesting values of $r$ occur when $\chi_r$ is at least a binomial and thus also has roots $\neq 0$ in $\C$. These values of $r$ correspond to the slopes of the  Newton polygon of $L$ and are also known as {\it dicritical values} or {\it weights} \cite{SvdP03} section 3.3, p.~90. The {\it dicritical weighted local exponents} of $L$ with respect to $r$ are defined as the non-zero roots of $\chi_r$ in $\C$. We set
\begin{align*}
\Omega_r&= \{\rho\in\C, \, \chi_r(\rho)=0\},\\
\Omega^*_r&=\Omega_r\sm\{0\}= \{\rho\in\C^*, \, \chi_r(\rho)=0\}.
\end{align*}
Here, $\Omega^*_r$ is non-empty if and only if $r$ is dicritical for $L$. Merkl then proves

%---------------------------------
%             LEMMA UNION OF EXPONENTS
%---------------------------------

\begin{lem} The number of classical local exponents of $L$ plus the number of dicritical weighted local exponents of $L$ with respect to rational weights $r>0$, both counted with their multiplicities, equals $n$, the order of $L$.
\end{lem}

%---------------------------------

In the case of a regular singularity, all local exponents are classical and no weighted local exponents appear. So we will assume henceforth that there is at least one weighted local exponent $\rho$, for some dicritical $r\in\Q^*$. Choose and then fix such a pair.

After these preparations, the first step in the algorithm is to replace in the differential equation $Ly=0$ the variable $y$ by 

\[ \exp(-{\rho\over r} x^{-r})y=e^{-{\rho\over r} x^{-r}}y.\]

This substitution corresponds to a conjugation of $L$ with the multiplication operator given by the indicated exponential function. If we write 
\[ L=\sum_{j=0}^n a_j(x)\delta_r^j\]
the conjugated operator is, see \cite{Merkl22} p.~13., given as 
\[\wt L=\sum_{j=0}^n\left(\sum_{k=j}^n {k\choose j}\rho^{k-j}a_j(x)\right)\delta_r^j.\]
It is then shown that the conjugation associated to a weighted local exponent $\rho$ of weight $r>0$ translates the weighted local exponents of $L$ by $\rho$, i.e., $\wt L$ has weighted local exponents $\sigma-\rho$ with respect to $r$ \cite{Merkl22}, Prop.~3.10, p.~29. In particular, the original $\rho$ becomes $0$ and is thus no longer dicritical for $\wt L$. Iterating this process of conjugation one arrives at a differential equation which has no dicritical weights at all. This is equivalent to saying that the final differential operator $L^*$ has a regular singularity at $0$. Thus the normal form theorem \ref{thm:nft0} applies to $L^*$ and produces as many linearly independent solutions  of $L^*y=0$ as its order indicates, using Theorem \ref{thm:sol0}. Tracing back the conjugations of $L$ and varying the algorithm over all dicritical weights $r$ and their weighted local exponents $\rho$, one ends up with a full basis of solutions of the original differential equation. This reproves in a constructive and systematic way Fabry's theorem about the existence and description of the solutions of irregular singular differential equations.

%---------------------------------
%             EXAMPLE  IRREGULAR
%---------------------------------

\begin{ex} The divergent series $y(x) =\sum_{k=0}^\infty k! x^{k+1}$ satisfies the second order equation 
\[Ly=x^3y''+(x^2-x)y'+y=0.\]
The initial form of $L$ at $0$ is given by the first order operator $L_0=-x\6+1$. Hence $0$ is an irregular singularity of $L$. The function $z(x)= \exp(-{1\over x})$ is a second solution of $Ly=0$; it is no longer a formal power series.
\end{ex}

%---------------------------------
%             APPARENT SINGULARITIES
%---------------------------------

{\bf Apparent singularities.} The formulas for the solutions of $Ly=0$ are somewhat complicated whenever the sets $\Omega$ of local exponents are not single valued. But if $\Omega=\{\rho\}$ has just one element $\rho$, i.e., no other local exponent of $L$ is congruent to $\rho$ modulo $\Z$, and if $\rho$ has multiplicity $m_\rho$, the respective solutions are simpler, of the form, for $0\leq i<m_\rho$,
\[\ds y_{\rho,i}(x)=x^\rho[f_i+\ldots +f_i\log(x)^i].\]
If some local exponents have multiplicity $\geq 2$ logarithms are forced to appear. If all local exponents are simple roots of the indicial polynomial, it may happen that no logarithms appear in the solutions. This situation is known as the presence of {\it apparent singularities}.

%---------------------------------
%             THM  APPARENT SINGULARITIES
%---------------------------------

\begin{thm} [Apparent singularities] Let $L\in\O[\6]$ be a differential operator with holomorphic coefficients and regular singularity at $0$. Assume that all local exponents are integers and simple roots of the indicial polynomial of $L$ at $0$, and write $L=L_0-T$ with initial form $L_0$ of $L$. If $\Im(T)\subseteq \Im(L_0)$ in $\O$, the local solutions of $Ly=0$ at $0$ are holomomorphic functions.
\end{thm}

%---------------------------------

\begin{proof} This is an immediate consequence of the proof of the normal form theorem, since in case $\Im(T)\subseteq \Im(L_0)$ no extensions of the differential operators to larger function spaces involving logarithms are needed. As the local exponents are integral, the assertion follows from the description of the solutions.
\end{proof}

%---------------------------------
%             GEVREY
%---------------------------------

{\bf Gevrey series.} By a theorem of Maillet, every power series solution $y(x)$ of an equation $Ly=0$ with holomorphic coefficients is a {\it Gevrey-series}, i.e., there exists an $m\in\N$ such that the $m$-th Borel transform 
\[\ds y(x)=\sum_{k=0}^\infty a_kx^k\mapsto \wt y(x)= \sum_{k=0}^\infty {a_k\over (k!)^m}x^k\]
of $y(x)$ converges \cite{Maillet03}. This result can also be seen as a consequence of the normal form theorem: It suffices to apply the norm estimates in part (ii) of the convergence proof to the series $\wt h(x) = \sum_{k=0}^\infty {a_k\over (k!)^m} x^k$ with $m=n-n'$, where $n'$ denotes again the order of the initial form $L_0$ of $L$ at $0$. Exploiting this one proves that the composition of the automorphism $v=u\inv$ of $\wh \O$ with the $m$-th Borel transform sends the solutions $x^\rho$ of $L_0y=0$, for $\rho\in\Z$ a local integer exponent of $L$, to a convergent power series $x^\rho \wt h(x)$. The key step is to see that the ratio
\[\ds {(\rr+k)^{\ul j}\over \chi_L(\rr+k+i-j)}=   {(\rr+k)^{\ul j}\over \sum_{\ell=0}^n c_{\ell,\ell}(\rr+k+i-j)^{\ul \ell}}\]
will be replaced by
\[\ds {(\rr+k)^{\ul j}\over \chi_L(\rr+k+i-j)}=   {(\rr+k)^{\ul j}\over \sum_{\ell=0}^{n'} c_{\ell,\ell}(\rr+k+i-j)^{\ul \ell}(k!)^m}\]
to obtain the required convergence. We omit the details.\qed

\section{Differential equations in positive characteristic} \label{sec:deqp}

\subsection{The lack of primitives in characteristic $p$} \label{subsec:lackp}

From now on let $\kk$ be a field of characteristic $p>0$. If we try to transfer the description of a basis of solutions of differential equations over $\C$ to fields of characteristic $p$, substantial obstructions occur, as the following example shows.

\begin{ex} \label{ex:zi}
(i) Let $n\in \N$ and let, for $S_{i,j}$ the Stirling numbers,
 $$L=\delta^n=(x\partial)^n=x^n\partial^n+S_{n,n-1}x^{n-1}\partial^{n-1}+S_{n,2}x^{n-2}\partial^{n-2}+\ldots +S_{n,1}x\partial+S_{n,0}.$$

If we interpret $L$ as a differential operator in $\C\osb x\csb [\partial]$ and solve the equation $Ly=0$ in $\C\orb x\crb [z]$, we obtain a full basis of solutions $ \{1, z, \ldots,\allowbreak z^{n-1} \} $ over $\C$. In characteristic $p$ the field of constants clearly contains $\kk\orb x^p\crb [z^p]$. So for $n>p$ the set $ \{1, z, \ldots,\allowbreak z^{n-1} \}$ cannot be a full basis of solutions, as $1$ and $z^p$ are linearly dependent over the field of constants. In some sense this boils down to the fact that a primitive of $z^{p-1}$ cannot be expressed in terms of $z^p$, in fact, $z^{p-1}$ does not have a primitive in $\kk\orb x\crb[z]$ at all.\\

(ii) Consider the Euler operator  
\[L=x^2\partial^2+x\partial+2.\]
Solving $Ly=0$ in characteristic $0$ we notice that the local exponents are given by $\sqrt{2}$ and $-\sqrt{2}$ and a basis of solutions is given by the ``functions'' $x^{-\sqrt{2}}$ and $x^{\sqrt{2}}$, which are defined in sectors without a branch cut of the logarithm around $0$. In $\F_7\orb x\crb $ the monomials $x^3$ and $x^4$ are solutions of the equation. However, in $\F_5$ no square root of $2$ exists and thus it is impossible to solve the Euler equation $Ly=0$ in $\F_5\orb x\crb[z]$.
\end{ex}

In order to resolve this issue in positive characteristic, we will construct in the next section a differential extension of $\kk\orb x\crb(z)$ which will contain a full basis of solutions for any linear differential operator with regular singularity at $0$. Regularity is again needed to ensure the existence of as many local exponents as the order of the differential equation indicates. The extension will overcome the two aforementioned difficulties: the lack of primitives of certain elements and the lack of solutions to Euler equations.

%---------------------------------
%             EULER PRIMITIVE CLOSURE
%---------------------------------

\subsection{The Euler-primitive closure $\mathcal{R}$ of $\kk\orb x\crb$} \label{subsec:closurep}

For each $\rho\in \kk$ let $t^\rho$ be a symbol. It will play the role of the monomial $x^\rho$ from before; if $\rho$ lies in the prime field of $\kk$ we may substitute $x$ for $t$ to recover the classical setting. We will call $\rho$ the exponent of $t$ in $t^\rho$. Further, let
$$\mathcal{R}= \bigoplus_{\rho \in \kk}t^\rho \kk(z_1, z_2, \ldots )\orb x\crb$$
be the direct sum of Laurent series in $x$ with coefficients in the field of rational functions over $\kk$ in countably many variables $z_i$, multiplied with the monomials $t^\rho$. We will simply write $\kk(z)$ and $\kk(z^p)$ instead of $\kk(z_1,z_2, \ldots)$ and $\kk(z_1^p,z_2^p, \ldots)$.

We consider $\mathcal{R}$ as a ring with respect to the obvious addition and the multiplication given by $$(t^\rho f)\cdot (t^{\sigma}g)=t^{\rho+\sigma}(f\cdot g)$$ for $\rho, \sigma\in \kk$, $f,g\in \kk[\z]\osb x\csb$. In other words, we form the group algebra of the additive group of $\kk$ over $\K(z)((x))$. We will write $t^0=1$ and accordingly have $(t^\rho)^p =t^{\rho p}=1$ and $\mathcal{R}^p=\kk(\z^p)\orb x^p \crb$.

Equip $\mathcal{R}$ with the derivation $\6=\6_R$ satisfying:
\begin{gather*}
\6 x = 1,\\
\6 t = t \frac{1}{x},\\
\6 t^\rho=\rho t^\rho \frac{1}{x},\\
\6z_1=\frac{1}{x}, \quad \6z_2=\frac{1}{x}\frac{1}{z_1},\quad \6 z_k=\frac{1}{x}\frac{1}{z_1\cdots z_{k-1}}=\frac{\6 z_{k-1}}{z_{k-1}}, \  k\geq 1.
\end{gather*}
This turns $\mathcal{R}$ into a differential ring.

The action of $\partial$ on $z_i$ is chosen to mimic the usual derivation of the $i$-fold composition $\log(\ldots (\log x )\ldots)$ of the complex logarithm with itself.
Indeed we have, writing $\log^{[i]}$ for the $i$-fold repetition of the logarithm
\[
\left(\log(x)^{[i]}\right)^\prime=\frac{1}{x\cdot \log x  \cdot \log(\log x )\cdots  \log(x)^{[i-1]}}.
\]

Similar constructions with iterated logarithms in positive characteristic were already considered  by Dwork \cite{Dwork90}, p.~752.

\begin{rem}
(i) The ring $\mathcal{R}$ is not an integral domain. Indeed, $(1+t+\ldots +t^{p-1})(1-t)=1-t^p=0$. Thus we are not able to form its quotient field and use the machinery of differential fields, as e.g.~Wronski's Lemma and the concept of a basis of solutions. Still, in the course of the next sections, we will be able to provide a precise description of the solutions of a differential equation $Ly=0$ in $\mathcal{R}$.

(ii) The derivation $\6$ leaves the summands of the direct sum $\mathcal{R}$ invariant, i.e., one has $\6\left(t^\rho \kk(\z)\orb x\crb\right)\subseteq t^\rho \kk(\z)\orb x\crb$. This is the reason for not simply defining $\partial t^\rho = \rho t^{\rho-1}$ but rather $\6 t^\rho=\rho t^\rho \frac{1}{x}$.

(iii) Note that the elements of $\mathcal{R}$ may have unbounded degree in each of the variables $z_i$, only the coefficient of a given power of $x$ has finite degree. This differs from the situation in characteristic $0$ where the exponent of the logarithm in a solution of the equation $Ly=0$ is bounded for each differential operator.

(iv) The doubly iterated logarithm $\log(\log x)$ of characteristic $0$ does not satisfy any homogeneous linear differential equation with holomorphic coefficients, but only the non-linear equation
$$
xy''+y'+x(y')^2=0.
$$
Alternatively, it satisfies the inhomogeneous equation
$$
x\log x y'=1
$$
in which the logarithm appears as a coefficient. In characteristic $p$ this reads as 
$$x z_1z_2'=1.$$
 
\end{rem}

 For elements of $\mathcal{R}$ the exponents of $x$ are integers, while the exponents of $t$ are elements of the field $\kk$ (formally, $t^\rho$ for $\rho\in \kk$ is just a symbol). However, we will see that the exponents of $x$ and $t$  interact in a certain way. We will use the following convention: In case that $\rho$ is in the prime field $\F_p$ of $\kk$, we may write $x^{\rho_\Z}$ for $x^\rho$ where $\rho_\Z\in \{0,1,\ldots, p-1\}$ is a representative of $\rho$. Conversely we may write $t^{k_\kk}$ for $t^k$ for some $k\in \Z$, where $k_\kk\in \F_p\subseteq \kk$ is the reduction of $k$ modulo $p$.

Before we proceed, we will determine the constants of $\mathcal{R}$. Denote by $\kk(\z^p)$ the subfield $\kk(z_1^p, z_2^p,\ldots)$ of $\kk(\z)$. A simple computation shows that monomials of the form $t^\rho z_i^p x^{mp-\rho}$, for any $\rho$ in the prime field $\F_p$ of $\kk$ and any $m\in \Z$, are annihilated by $\6$. And, actually, these monomials already yield the entire field of constants, namely,

\begin{prop}\label{prop:const}
The ring of constants of $(\mathcal{R},\6)$ is $$\mathcal{C}:=\bigoplus_ {\rho \in \F_p}t^\rho x^{p-\rho}\kk(\z^p)\orb x^p\crb,$$ where $\F_p$ denotes the prime field of $\kk$. Moreover, $\mathcal{C}$ is a field.
\end{prop}
\begin{proof}
Let $f \in \bigoplus_{\rho \in \kk} t^\rho \kk(\z)\orb x\crb$ and assume that 
$$
\6 f=0.
$$
Taking derivatives in $\mathcal{R}$ preserves the summands of the direct sum, so it suffices to find constants of the form $t^\rho h$ for some $\rho \in \kk$ and $h\in \kk(\z)\orb x\crb.$

Fix some $\rho \in \kk$. As for all $k\in \Z$ the derivation $\6$ maps $t^\rho \kk(\z)x^k$ into $t^\rho \kk(\z)x^{k-1}$ by definition, it further suffices to find constants of the form $t^\rho h x^k$, where $h\in \kk(\z)$. Therefore we are reduced to search for elements $t^\rho h x^k$ of $\mathcal{R}$ with $\6(t^\rho h x^k)=0.$ Write $h=g_1/g_2$ for $g_1, g_2\in \kk[\z]$. Then  $\6(t^\rho h x^k)=0$ is equivalent to $\6(t^\rho g_1g_2^{p-1} x^k)=0$, as $g_2^p$ is a constant. So without loss of generality we may assume that $h\in \kk[\z]$ is a polynomial. We expand:
\begin{equation} \label{eq:vanconst}
0=\6(t^\rho h x^k)=t^\rho((\6 h) x + (k+\rho) h)x^{k-1}.
\end{equation}

Let $l$ be minimal such that $h\in\kk[z_1,\ldots, z_l]$. Consider one monomial $\z^\alpha=z_1^{\alpha_1}\cdots z_l^{\alpha_l}$ of $h$, whose 
exponent $\alpha$ is maximal among the monomials of $h$ with respect to the component-wise ordering of $\N^l$. Taking the derivative $\partial$ of a monomial in $\z$ decreases the exponents of at least one of the $z_i$ and does not increase the other. It therefore yields a sum of smaller monomials with respect to the chosen ordering.

Thus, in $x\6h$ the coefficient of $z^\alpha$ vanishes by the maximality of the exponent. If we compare coefficients of $t^\rho x^{k-1}z^\alpha$ in Equation (\ref{eq:vanconst}) we get $k+\rho=0$. So it follows that $\rho\in \F_p$ and that $k\equiv \rho \mod p$. Moreover, we see from Equation (\ref{eq:vanconst}) that $\6 h=0$. This is clearly equivalent to $h\in \kk[\z ^p]$. Together with the reductions from above this proves that the ring of constants of $\mathcal{R}$ is indeed
$$
\bigoplus_ {\rho \in \F_p}t^\rho x^{p-\rho}\kk(\z^p)\orb x^p\crb.
$$

Finally, we show that $\mathcal{C}$ is a field. Let
$$
f=\sum_{\rho \in \F_p}t^\rho x^{p-\rho} f_\rho\in \mathcal{C},
$$
where $f_\rho\in \kk( \z^p)\orb x^p\crb$. Then we have
$$f^p=\sum_{\rho \in \F_p}t^{p\rho} x^{p^2-p\rho} f_\rho^p=\sum_{\rho \in \F_p} x^{p^2-p\rho} f_\rho^p\in \kk(\z^p)\orb x^p\crb,$$
where $f_\rho^p\in \kk( \z^{p^2})\orb x^{p^2}\crb$. The element $f^p$ vanishes precisely if $f_\rho$ vanishes for all $\rho \in \F_p$, as the exponents of $x$ in each of the summands are from a different residue class modulo $p^2$. Thus, $f^p$ is a unit for all $f\neq 0$ and we see that $(f^{p-1})(f^p)^{-1}$ is an inverse to $f$. 
\end{proof}

{\bf Example~\ref{ex:zi} (cont.)} Let us come back to Example \ref{ex:zi} (i) with $k=p+1$ and the operator $L=(x\partial)^{p+1} \in \kk[x][\partial]$. In $\mathcal{R}$ we have
$$
(x\6)^{p+1}(z_1^pz_2)=0.
$$
So we have found another solution to the equation $Ly=0$. This completes a basis of a $p+1-$ dimensional vector space of solutions over the constants of $\mathcal{R}$, namely $\{1,z_1^1, z_1^2,\ldots ,z_1^{p-1}, z_1^{p}z_2\}$, as those elements are $\mathcal{C}$-linearly independent.\\

For the Euler operator $L=x^2\partial^2+x\partial+2\in \F_5[x][\partial]$ from Example \ref{ex:zi} (ii) we can also find a basis of solutions in $\mathcal{R}$ over $\mathcal{C}$. It is given by the monomials $t^{\omega}$ and  $t^{-\omega}$, where $\omega\in \F_{25}$ is a square root of $2$.\\

From what we have seen it is reasonable to call $\mathcal{R}$ the {\it Euler-primitive closure} of $\kk\orb x\crb$.
\subsection{Extensions of Euler operators to the ring $\mathcal{R}$}\label{subsec:eulerp}
Our goal now is to prove that Euler operators admit ``enough'' solutions in  $\mathcal{R}= \bigoplus_{\rho \in \kk}t^\rho \kk(\z)\orb x\crb$ and then to compute these solutions. For this we first investigate how Euler operators act on monomials $t^\rho z^\alpha x^k$, see Lemma \ref{lem:66mon}. For a multi-index $\alpha \in \Z^{(\N)}=\{(\alpha_i)_{i\in \N}\mid \alpha_i=0 \text{ for almost all } i\}$ we write $\z^\alpha$ for $z_1^{\alpha_1}\cdots z_n^{\alpha_n}$, if $\alpha_j=0$ for $j>n$. We define a partial ordering on $\Z^{(\N)}$ by $\beta \prec_e \alpha$ if
$$e(\beta):=\overline{\beta}_1+p\overline{\beta}_2+p^2\overline{\beta}_3+\ldots<\overline{\alpha}_1+p\overline{\alpha}_2+p^2\overline{\alpha}_3+\ldots=:e(\alpha),
$$
where $\overline{\beta}_i, \overline{\alpha}_i\in\{0,1,\ldots, p-1\}$ are chosen such that $\beta_i\equiv \overline{\beta}_i\mod p$ respectively $\alpha_i\equiv \overline{\alpha}_i\mod p$.
In other words $\prec_e$ is induced by the inverse lexicographic ordering on $\F_p^{(\N)}$ via the element-wise reduction modulo $p$ of elements of $\Z^{(\N)}$.

We also write $\z^\beta\prec_e \z^\alpha$ if $\beta\prec_e \alpha$.

\begin{lem}\label{lem:order}
Let $\alpha \in \Z^{(\N)}$. Then $(x\partial) z^\alpha$ is a sum of monomials that are smaller than $z^\alpha$ with respect to $\prec_e$ and there is exactly one summand $z^\gamma$ with $e(\gamma)=e(\alpha)-1$. In particular, $e(\alpha)$ is the minimal number $j$ such that $(x\partial)^j(z^\alpha)=0$.
\end{lem} 
\begin{proof}
Let $\alpha=(\alpha_1,\alpha_2,\ldots)$. We compute:
$$\6 \z^\alpha=\frac{1}{x}\sum_{i=1}^t\alpha_i\underbrace{z_1^{\alpha_1-1}z_2^{\alpha_2-1}\cdots z_i^{\alpha_i-1}z_{i+1}^{\alpha_{i+1}}\cdots z_t^{\alpha_t}}_{=: \z^{\gamma(i)}}.
$$
If $\alpha_i\not \equiv 0\mod p$, then clearly $\gamma(i)\prec_e \alpha$, otherwise its coefficient in $(x\partial) \z^\alpha$ vanishes. A fast computation shows that if $j$ is the least index, such that $\alpha_j\neq 0$, then $e(\gamma(j))=e(\alpha)-1$. Moreover, $e(\gamma(j))<e(\alpha)-1$ for all other $j$. This proves in particular that $e(\alpha)$ is the minimal number $j$ such that $(x\6)^jz^\alpha=0$. 
\end{proof}

Let $s$ be a variable and $k\in \N$. We define the {\it $j$-th Hasse derivative} or {\it divided derivative} of $s^k$ by $(s^k)^{[j]}=\binom{k}{j}s^{k-j}$; extend it linearly to $\kk[s]$ \cite{Jeong11}. We will apply it below to the indicial polynomial $\chi_L$ of an operator $L$, viewed as a polynomial in the variable $s$. The next three lemmata are, as in the case of characteristic zero, inspired by Frobenius' ``differentiation with respect to local exponents'' \cite{Frobenius73}. See Lemmata \ref{lem:substitutionz} and \ref{lem:kernelL0} for the corresponding results in characteristic zero.
\begin{lem}
Let $k,l\in \N$. Then we have
$$
(s^{\ul k})^{[l]}+(s^{\ul k})^{[l+1]}(s-k)=(s^{\ul{ k+1}})^{[l+1]}.
$$\end{lem}

\begin{lem} \label{lem:66mon}
Let $j\in \N, k\in \Z, \alpha \in \Z^{(\N)}$. Then we have
$$
\6^j(t^s x^k \z^\alpha)=t^s x^{k-j} \left((s+k)^{\ul j} \z^\alpha + ((s+k)^{\ul j})^{[1]}x\6\z^\alpha +\ldots + ((s+k)^{\ul j})^{[j]}(x\6)^j \z^\alpha\right).
$$
\end{lem}
\begin{proof}
The proof uses induction on $j$. For $j=0$ the claim is obvious. Assume now the formula holds for some $j\geq 0$. Applying $\6$ yields 
\begin{align*}
&\6^{j+1}(t^s x^k \z^\alpha)=\6\left(t^s x^{k-j} \left((s+k)^{\ul j} \z^\alpha + ((s+k)^{\ul j})^{[1]}x^{1}\6\z^\alpha +\ldots + ((s+i)^{\ul j})^{[j]}(x\6)^j \z^\alpha\right)\right)\\
&\ =t^s x^{k-j-1}(s+k-j) \left((s+k)^{\ul j} \z^\alpha + ((s+k)^{\ul j})^{[1]}x\6\z^\alpha +\ldots + ((s+k)^{\ul j})^{[j]}(x\6)^j \z^\alpha\right)+ \\
&\qquad + t^s x^{k-j}\left((s+k)^{\ul j} \6 \z^\alpha + ((s+k)^{\ul j})^{[1]}\6(x\6)\z^\alpha+\ldots + ((s+k)^{\ul j})^{[j]}\6(x\6)^j \z^\alpha\right)\\
&\ =t^s x^{k-j-1}\left((s+k-j)(s+k)^{\ul j}+ \left((s+k-j)((s+k)^{\ul j})^{[1]}+(s+k)^{\ul j}\right)x\6\z^\alpha+\ldots \right)\\
&\ =t^s x^{k-j-1} \left((s+k)^{\ul{j+1}} \z^\alpha + ((s+k)^{\ul{j+1}})^{[1]}x\6\z^\alpha +\ldots + ((s+k)^{\ul{j+1}})^{[j+1]}(x\6)^{j+1} \z^\alpha\right),
\end{align*}
where we have used the previous lemma in the last step.
\end{proof}
From this we get:
\begin{lem}  \label{lem:monomials}
Let $L$ be an Euler operator of order $n$ with indicial polynomial $\chi_L$. Then for any $\alpha\in \Z^{(\N)}$, $k\in \Z$ and $\rho \in \kk$ we have
\begin{equation} \label{eq:expansion}
L(t^\rho x^{k}\z^\alpha)=t^{\rho}x^{k}\left(\chi_L(\rho+k) \z^\alpha+\chi_L'(\rho+k)x\6(\z^\alpha)+\ldots + \chi_L^{[n]}(\rho+k)(x\6)^n(\z^\alpha)\right).
\end{equation}
\end{lem}

For a field $K$ of characteristic $0$ a polynomial $q\in K[s]$ has a $j$-fold root at $\beta\in \overline{K}$ if and only if the first $j-1$ derivatives of $q$ vanish in $\beta$, but the $j$-th derivative does not. This very statement is false in characteristic $p$, but if one replaces derivatives with Hasse derivatives it holds true.

\begin{lem} \label{lem:vanishingHasse}
Let $q\in \kk[s]$ be a polynomial. Then $a$ is an $j$-fold root of $q$ if and only if $q^{[i]}(a)=0$ for $i<j$, but $q^{[j]}(a)\neq 0$.
\end{lem}

With these results we can finally solve Euler equations in the ring $\mathcal{R}$. We prove that, similar to the complex case in Lemma \ref{lem:solutionsL0}, the solutions form a vector space of dimension $n$ over the constants $\mathcal{C}\subseteq \mathcal{R}$.
\begin{prop} \label{prop:Euler}
Let $L$ be an Euler operator of order $n$ and let $\Omega:=\{\rho_1,\ldots, \rho_k\}$  be the set of local exponents of $L$ at $0$ with multiplicities $m_{\rho_1},\ldots ,m_{\rho_k}$. The solutions of $Ly=0$ in $\RR$ form a  $\mathcal{C}$-subspace of dimension $n$. A basis is given by
$$y_{\rho, i}:=t^{\rho} \z^{i^*}, \quad \rho \in \Omega,  \ i<m_{\rho},$$
where 
$$i^*=(i, \lfloor i/p \rfloor, \lfloor i/p^2 \rfloor, \lfloor i/p^3 \rfloor,\ldots)\in \Z^{(\N)}. $$
\end{prop}

Before we prove the proposition let us consider an example.
\begin{ex}
Consider the differential operator $L=x^6\partial^6 +x^4\partial^4+x^3\partial^3+x^2\partial^2\in \F_2[x][\partial]$ with indicial polynomial $\chi_L(s)=(s-1)^5s$. As the operator has order $6$ one expects $6$ solutions of $Ly=0$, independent over $\mathcal{C}$. The proposition asserts that a basis is given by \[1, x, x z_1, x z_1^2z_2, x z_1^3z_2, x z_1^4z_2^2z_3.\] Indeed, one easily verifies that all these monomials are solutions and are $\mathcal{C}$-linearly independent.
\end{ex}

\begin{proof}
The operator $L$ is $\mathcal{C}$-linear and maps $t^\rho x^k\kk( \z )$ into itself. Therefore it suffices to find solutions of $Ly=0$ of the form $t^\rho f(\z) x^k$, where $f\in \kk( \z)$. Further we can argue similar as in Proposition \ref{prop:const}: we write $f=g_1/g_2$ for $g_1, g_2\in \kk[\z]$. If $t^\rho f(\z) x^k$ is a solution, then so is
$$
g_2(\z)^p\left(t^\rho f(\z) x^k\right)=t^\rho g_1(\z)g_2(\z)^{p-1} x^k,
$$
as $g_2^p\in \kk[\z^p]\subseteq \mathcal{C}$. So we may assume without loss of generality that $0\neq f\in \kk[\z]$.

Let $z^\alpha$ be the largest monomial of $f(\z)$ with respect to the ordering $\prec_e$. By Lemma \ref{lem:monomials} and the linearity of $L$ we obtain 
$$
L(t^\rho f(\z) x^k)=t^\rho\left(\chi_L(\rho+k)f(\z)+\chi_L^{[1]}(\rho+k)(x\6)f(\z)+\ldots +\chi_L^{[n]}(\rho+k)(x\6)^nf(\z)\right).
$$
Hence $L(t^\rho f(\z) x^k)$ vanishes if and only if 
$$
\chi_L(\rho+k)f(\z)+\chi_L^{[1]}(\rho+k)(x\6)f(\z)+\ldots +\chi_L^{[n]}(\rho+k)(x\6)^nf(\z)
$$
vanishes. We compare the coefficients of monomials in $\z$ starting with the largest. All appearing monomials are smaller than or equal to $z^\alpha$ by Lemma \ref{lem:order} and for all monomials $\z^\gamma$ in the summand $\chi_L(\rho+k)(x\partial)^j$ we have $e(\gamma)\leq e(\alpha)-j$. So in order for the sum to vanish, $\chi_L(\rho+k)$ has to vanish by comparing coefficients of $ \z^\alpha$. Further, by comparing coefficients of the next smaller monomials, we obtain $\chi_L^{[1]}(\rho+k)=0$ or $(x\partial)z^\alpha=0$, i.e. $e(\alpha)=1$. Inductively we obtain that the sum vanishes, if and only if $\chi_L^{[\ell]}(\rho+k)\neq 0$ implies that $e(\alpha)<\ell$. Put differently, by Lemma \ref{lem:vanishingHasse}, if $\rho+k$ is a local exponent of $L$ of multiplicity $m_{\rho+k}$, then $e(\alpha)<m_{\rho+k}.$  Thus we can give a complete description of the elements in the kernel of $L$. They are of the form $t^\rho x^k z^\alpha$, where $e(\alpha)<m_{\rho+k}$. 

A quick calculation using Lemma \ref{lem:order} shows that the last condition is fulfilled for multi-indices, whose entries differ by multiples of $p$ from $i^*$ for $i=0,\ldots, m_{\rho+k}-1$. This shows on the one hand that the elements $y_{\rho, i}$ are indeed solutions of $Ly=0$. On the other hand, the elements $y_{\rho, i}$ are chosen such that $\rho$ ranges over all local exponent of $L$ exactly once. For $\rho+k$ to be a local exponent, i.e., a zero of $\chi_L$, we may add multiples of $p$ to $k$, or subtract an element of the prime field from $\rho$ and add it to $k$. Those transformations can be realized by multiplying a solution $t^\rho x^k f(\z)$ by an element from $\mathcal{C}$. So indeed, all solutions of $Ly=0$ are linear combinations of the elements $y_{\rho, i}$; that is they generate the $\mathcal{C}$-vector space of solutions.

Assume now that a $\mathcal{C}$-linear relation between the solutions $y_{\rho, i}$ exists. Let 
$$\mathcal{D}:=\bigoplus_ {\rho \in \F_p}t^\rho x^{p-\rho}\kk[\z^p]\osb x^p\csb.
$$
As  $\mathcal{C}={\rm Quot}\mathcal{D}$, it suffices to consider a relation with coefficients in $\mathcal{D}$. Let $\Omega=\bigsqcup_j \Omega_j$ be the set of all local exponents, where two local exponents $\rho, \sigma$ are in the same subset $\Omega_j$ if and only if their difference is in the prime field.  Assume that 
$$
\sum_j \sum_{\substack{\rho_\in \Omega_j\\ i < m_\rho}} y_{\rho, i} \cdot d_{\rho, i}=0
$$ 
for some $d_{\rho, i}\in \mathcal{D}$. As the exponents of $t$ of elements of $\mathcal {D}$ are in the prime field of $\kk$, it follows that for each $j$ the sum 
$$
\sum_{\substack{\rho_\in \Omega_j\\ i< m_\rho}} y_{\rho, i}\cdot  d_{\rho, i}
$$
vanishes. So it suffices to focus on relations between solutions corresponding to local exponents in the same set set $\Omega_j$. Without loss of generality $\Omega_j=\F_p$, the prime field of $\kk$. We consider now a relation of the form
$$
\sum_{\substack{\rho \in \F_p\\ i<m_\rho}} y_{\rho, i}\cdot  d_{\rho, i}=0.
$$
Without loss of generality we may assume that at least one of the constants $d_{\rho, i}$ has order $0$ in $x$ and let $f_{\rho, i}\in \kk[\z^p]$ be its constant term. Taking the coefficient of the monomial with smallest degree with respect to $x$ in the sum above, we obtain a relation of the form
$$
\sum_{\substack{\rho \in \F_p\\ i<m_\rho}} t^\rho z^{\alpha_i} \cdot f_{\rho, i}=0.
$$
This sum vanishes if and only if the summand for each $\rho\in \F_p$ vanishes. Furthermore the  multi-exponents $i^*=(i, \lfloor i/p\rfloor, \lfloor i/p^2\rfloor,\ldots)$ are defined such that no two of them differ by multiples of $p$ in every component. Thus $f_{\rho, i}=0$ for all $\rho$ and $i$, as required.

Finally, note that $\sum_{\rho \in \Omega}m_\rho=n$, as $\chi_L$ is a polynomial of degree $n$. So the dimension of the space of solutions is indeed $n$.
\end{proof}

\subsection{The normal form theorem in positive characteristic} \label{subsec:nftp}
We have seen that a basis of solutions of Euler equations is of a very special form. It is not to be expected that solutions of general differential equations with regular singularities are similarly simple. %We now try to find a description of the solutions of  differential equations $L$ based on the basis of monomial solutions of their initial form $L_0$. 
In the following let $\rho$ be a fixed local exponent of $L$ at $0$ of multiplicity $m_\rho$. We define a function $\xi:\N\to \N$
$$
\xi(0)=m_\rho,\qquad \xi(k+1)=\xi(k)+m_{\rho + k +1},
$$
where $m_{\rho + j}=0$ if $\rho + j$ is not a root of the indicial polynomial. In other words
$$
\xi(k)=m_\rho+m_{\rho+1}+\ldots + m_{\rho+k}.
$$
Note here that if $k>p$ the summand $m_{\rho}$ appears at least twice in the sum. Moreover we define the {\it $\rho$-function space $\mathcal{F}=\mathcal{F}^\rho_L$ associated to $L$} as
$$
\mathcal{F}^\rho_L:=t^\rho\sum_{k=0}^\infty \bigoplus_{\alpha \in \mathcal{A}_k}\kk \z^\alpha x^k,
$$
where 
$$
\mathcal{A}_k:=\left\{\alpha \in \N^{(\N)}\middle | \alpha_1<\xi(k), \alpha_{j+1}\leq\alpha_j/p \quad \text{for all}\ j\in \N\right\}
$$
is a finite subset of $\N^{(\N)}$. Note that $\mathcal{F}$ only depends on the initial form of the differential operator $L$, more precisely, only on the multiplicity of all local exponents of $L$ that differ from $\rho$ by an element of the prime field of $\kk$.
\begin{ex}
Consider the differential operator $L=x^3\partial ^3+2x^2\partial^2+\widetilde L\in \F_3\osb x\csb [\partial]$, where $\widetilde L\in \F_3\osb x\csb[\partial]$ has positive shift. The local exponents of $L$ are $0$ with multiplicity $2$ and $1$ with multiplicity $1$. The monomials in $\mathcal{F}^0_L$ are depicted below in Figure \ref{fig:monomials}.
\end{ex}

\begin{figure}[htb!] 
\centering
\includegraphics{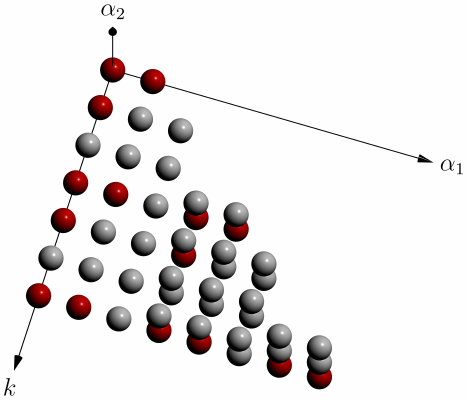}
\caption{The set of exponents $(k, \alpha_1, \alpha_2)$ of monomials $x^kz_1^{\alpha_1}z_2^{\alpha_2}$  in $\mathcal{F}^0_L$ with $k\leq 6$, exponents of monomials in $\Ker(L_0)$ in red. They are $1, z_1, x, x^3, x^3z_1, x^3z_1^3, x^3z_1^4, x^4, x^4z_1^3,\allowbreak x^6, x^6z_1, x^6z_1^3, x^6z_1^4, x^6z_1^6, x^6z_1^7$.} \label{fig:monomials}
\end{figure}

\begin{lem} \label{lem:closed}
Let $L\in \kk\osb x\csb [\partial]$ be a linear differential operator and let $\rho$ be one of its local exponents. The space $\mathcal{F}^\rho_L=\mathcal{F}$ is invariant under all differential operators with non-negative shift. In particular we have $L\mathcal{F}\subseteq \mathcal{F}$.
\end{lem}
\begin{proof}
We can rewrite any differential operator with non-negative shift in terms of the operator $\delta=x\partial$ instead of $\partial$, where the base change between $x^n\partial^n$ and $\delta^n$ is given by the Stirling numbers, see Remark \ref{rem:stirling}. So we investigate the action of $\delta$ on a monomial $t^\rho x^i z^\alpha\in \mathcal{F}$, where $\alpha=(\alpha_1,\ldots, \alpha_n)\in \N^{(\N)}.$ We compute as in Lemma \ref{lem:order}
$$
\delta (t^\rho x^k z^\alpha) = x\6(t^\rho x^k z^\alpha)=t^\rho x^{k} \left( (k+\rho)z^\alpha + \sum_{j=1}^n\alpha_j z_1^{\alpha_1-1}\cdots z_j^{\alpha_j-1}z_{j+1}^{\alpha_{j+1}}\cdots z_n^{\alpha_n}\right).
$$
We want to show that all exponents of monomials with non-zero coefficient in the sum above are in $\mathcal{A}_k$. It is clear that $\alpha\in \mathcal{A}_k$ by assumption, so it remains to prove that if $\alpha_j\not \equiv 0\mod p$ then $$(\alpha_1-1,\ldots, \alpha_j-1, \alpha_{j+1}, \ldots, \alpha_n)\in \mathcal{A}_k$$ for $j=1,\ldots, n$. If $\alpha_{l+1}\leq\alpha_l/p$ then also $\alpha_{l+1}-1\leq(\alpha_l-1)/p$ for $l<j$. It remains to show that $\alpha_{j+1}> (\alpha_j-1)/p$ implies $\alpha_j\equiv 0 \mod p$. For this we see that from
$$
\alpha_j-1< p\alpha_{j+1}\leq\alpha_j
$$ 
 it follows indeed that $p$ divides $\alpha_j = p\alpha_{j+1}$.
\end{proof}

\begin{lem} [cf.~Lemma \ref{lem:image}]\label{lem:surjectivity}
Let $L_0\in \kk\osb x\csb[\partial]$ be an Euler operator with local exponent $\rho$ and associated $\rho$-function space $\mathcal{F}$. Then $L_0(\mathcal{F})=\mathcal{F} x$.
\end{lem}
\begin{proof}
First we show that any monomial in $\mathcal{F}$ gets mapped to $\mathcal{F} x$ under $L_0$. Let $t^\rho x^k z^\alpha\in \mathcal{F}$. By Lemma \ref{lem:monomials} we have 
$$
L_0(t^\rho x^{k}\z^\alpha)=t^{\rho}x^{k}\left(\chi_L(\rho+k) \z^\alpha+\chi_L'(\rho+k)x\6(\z^\alpha)+\ldots + \chi_L^{[n]}(\rho+k)(x\6)^n(\z^\alpha)\right).
$$
By Lemma \ref{lem:closed} this expression is contained in $\mathcal{F}$. The first $m_{\rho+k}$ summands of the sum vanish due to Lemma \ref{lem:vanishingHasse}. In the remaining summands $x\6$ is applied at least $m_{\rho+k}$ times to $\z^\alpha$, decreasing the exponent of $z_1$ by at least $m_{\rho+k}$. Thus for each monomial with non-zero coefficient in 
$$
\chi_L^{[m_{\rho+k}]}(\rho+k)(x\6)^{m_{\rho+k}}(\z^\alpha)+\ldots + \chi_L^{[n]}(\rho+k)(x\6)^n(\z^\alpha)
$$ 
the exponents of $\z$ are in $\mathcal{A}_{k-1}$ and thus $L_0(t^\rho x^{k}\z^\alpha)\in \mathcal{F} x$.

Now we show that every monomial of $\mathcal{F} x$ is in the image of $\mathcal{F}$ under $L_0$. We proceed by induction on $e(\alpha)=\overline{\alpha_1}+p\overline{\alpha_2}+p^2\overline{\alpha_3}+\ldots$ Let $t^\rho x^{k+1} z^\alpha\in \mathcal{F} x$; that is $\alpha\in \mathcal{A}_k$. Assume that $\rho+k+1$ is an $\ell$-fold root of $\chi_L$, where $\ell$ is set equal to $0$ if $\rho+k+1$ is not a root at all. We define an element $\beta\in \mathcal{A}_{k+1}$ such that $L_0(t^\rho x^{k+1}\z^\beta)=t^\rho x^{k+1}\z^\alpha+r$, where $r$ is a sum of smaller monomials with respect to $\prec_e$. Set $$\beta_1=\alpha_1+\ell,\qquad \beta_j=\alpha_j+\lfloor \beta_{j-1}/p \rfloor-\lfloor \alpha_{j-1}/p\rfloor.$$

As $t^\rho x^k \z^\alpha\in \mathcal{F}$, we have $\alpha_1<\xi(k)$ and therefore $\beta_1=\alpha_1+\ell<\xi(k+1).$ Moreover, we know that $\alpha_j\leq \lfloor \alpha_{j-1}/p\rfloor$ and therefore also $\beta_j=\alpha_j+\lfloor\beta_{j-1}/p\rfloor-\lfloor\alpha_{j-1}/p\rfloor\leq  \beta_{j-1}/p$. By construction we have $\beta_1=\alpha_1+l$ and thus $\beta_1<\xi(k+1)$. Altogether this proves $\beta \in \mathcal{A}_{k+1}$.

Finally we show  that $L_0(t^\rho x^{k+1}\z^\beta)$ is of the desired form. Again  by Lemma \ref{lem:monomials} we have
$$
L_0(t^\rho x^{k+1}\z^\beta)=t^{\rho}x^{k+1}\left(\chi_L(\rho+k+1) \z^\beta
% +\chi_L'(\rho+k+1)x\6(\z^\beta)
+\ldots + \chi_L^{[n]}(\rho+k+1)(x\6)^n(\z^\beta)\right)
$$
As $\rho +k +1$ has multiplicity $\ell$ as a zero of $\chi_L$, the first $\ell$ summands of this expansion vanish, according to Lemma \ref{lem:vanishingHasse}. If one expands the further summands using the Leibniz rule one gets a sum of monomials of the form $c_\gamma t^\rho x^{k+1} \z^\gamma$, with $c_\gamma \in \kk$. The exponents $\gamma$ are in $\mathcal{A}_k$ and by Lemma \ref{lem:order} we have $e(\gamma)\leq e(\beta)-\ell=e(\alpha)$. Only one of these summands fulfils $e(\gamma)=e(\alpha)$. It is of the form $c_\alpha t^\rho x^{k+1} \z^\alpha$ by construction. Now by the induction hypothesis, all other summands are in the image of $\mathcal{F}$ under $L_0$; they are in $\mathcal{F} x$ because of Lemma \ref{lem:closed}. Thus, $t^\rho x^{k+1}\z^\alpha\in L_0(\mathcal{F})$, which concludes the proof.
\end{proof}
 
\begin{rem} \label{rem:constructive}
The proof of the surjectivity of $L_0$ is constructive: For each monomial $t^\rho x^kz^\alpha$ in $\mathcal{F} x$ one constructs a monomial $t^\rho x^kz^\beta$ in $\mathcal{F}$, such that $L_0(t^\rho x^kz^\beta)=c t^\rho x^kz^\alpha+r$, where $r$ is a sum of smaller monomials. If $r=0$ we divide by $c$ and are done. Otherwise we iterate the construction for all monomials in $r$ and subtract the monomials obtained this way from $c^{-1} t^\rho x^kz^\beta$. After at most $e(\alpha)$ steps $r=0$ and we have constructed an element of $\mathcal{F}$ which is sent to $t^\rho x^kz^\alpha$ by $L_0$. 
\end{rem}

\begin{lem}
The kernel of the restriction of $L_0$ to $\mathcal{F}=\mathcal{F}^\rho_L$ is topologically spanned over $\kk$ by monomials of the form $t^\rho x^kz^\alpha$, where
$$
\alpha\in \mathcal{A}_k=\left\{\alpha \in \N^{(\N)}\middle | \alpha_1<\xi(k), \alpha_{j+1}\leq \alpha_j/p \quad \text{for all} \ j\in \N\right\}
$$
with $e(\alpha)<m_{\rho+k}$. Consequently, a direct complement $\mathcal{H}$ of $\Ker\, L_{0}\vert_{\mathcal{F}}$ is topologically spanned by monomials of the form $t^\rho x^kz^\alpha$, where $\alpha \in \mathcal{A}_k$ with $e(\alpha)\geq m_{\rho+k}.$
\end{lem}
\begin{proof}
We have seen that $e(\alpha)$ is the least number $k$, such that $(x\partial)^k\z^\alpha=0$. So every monomial $t^\rho x^kz^\alpha$ with $e(\alpha)<m_{\rho+k}$ is in the kernel of $L_0$ according to Lemma \ref{lem:monomials}. Arguing as in the proof of Proposition \ref{prop:Euler} we see that those elements indeed span $\ker L_{0}\vert_\mathcal{F}$.
\end{proof}

%---------------------------------
%             NFT CHAR p
%---------------------------------

Now we are ready to state and prove the normal form theorem.
\begin{thm}[Normal form theorem in positive characteristic ]\label{thm:nft}
Let $\kk$ be an algebraically closed field of characteristic $p$. Let $L\in \kk\osb x\csb[\partial]$ be a differential operator with initial form $L_0$ and shift $\tau=0$ acting on $\mathcal{R}= \bigoplus_{\rho \in \kk}t^\rho \kk(z)\orb x\crb$. Let $\rho$ be a local exponent of $L$ at $0$ and $\mathcal{F}=t^\rho\sum_{k=0}^\infty \bigoplus_{\alpha \in \mathcal{A}_k}\kk \z^\alpha x^k\subset \RR$ the associated $\rho$-function space.
\begin{enumerate}[(i)] 
\item The map $L_0\vert_\mathcal{H}:  \mathcal{H} \to \mathcal{F} x$ is bijective and the composition of its inverse $(L_0\vert_{\mathcal{H}})^{-1}:\mathcal{F} x\to \mathcal{H}$ composed with the inclusion $\mathcal{H}\subseteq \mathcal{F}$ defines a $\mathcal{C}$-linear right inverse $S:\mathcal{F} x\to \mathcal{F}$ of $ L_0$.
\item Let $T=L_0-L:\mathcal{F}\to \mathcal{F} x$. Then the map 
$$
u=\Id_{\mathcal{F}}-S\circ T:\mathcal{F}\to \mathcal{F}
$$
is a continuous $\mathcal{C}$-linear automorphism of $\mathcal{F}$ with inverse $v=\sum_{k=0}^\infty (S\circ T)^k:\mathcal{F}\to \mathcal{F}$.
\item The automorphism $v$ of $\mathcal{F}$ transforms $L$ into $L_0$, i.e.,
$$
L\circ v=L_0.
$$
\end{enumerate}
\end{thm}

\begin{proof} For (i) note that by Proposition \ref{lem:surjectivity} the map $L_0:\mathcal{F}\to \mathcal{F} x$ is surjective and thus the restriction to a direct complement of its kernel is bijective. Clearly $S$ then defines a right inverse of $L_0$. One easily checks that the construction of preimages of $L_0$ mentioned in Remark \ref{rem:constructive} is $\mathcal{C}$-linear.

The assertions (ii) and (iii) are an application of the Perturbation Lemma~\ref{lem:pert}. We view elements of $\mathcal{F}$ as power series in $x$ and equip $\mathcal{F}$ with the $x$-adic topology, which turns it into a complete metric space. The operator $T$ has positive shift by definition and thus increases the order in $x$ of a monomial $t^\rho x^kz^\alpha$ and thus of any element of $\mathcal{F}$. The operator $S$ maintains the order in $x$ of a monomial as $L_0$ does so. So the composition $S\circ T$ increases the order of any element. Furthermore, $T$ maps $\mathcal{F}$ to $\mathcal{F} x=\Im (L_0)$. So we may apply the perturbation lemma and the claim follows.
\end{proof}

\subsection{Solutions of regular singular equations}\label{subsec:solp}
The normal form theorem allows us to describe all solutions of differential equations with regular singularities.
\begin{thm} \label{thm:solp}
Let $L\in \kk\osb x\csb [\partial]$  be a linear differential operator with regular singularity at $0$ acting on $\mathcal{R}$. Let $\rho\in \kk	$ be a local exponent of $L$. Denote by $u_\rho:\mathcal{F}^\rho_L\to \mathcal{F}^\rho_L$ the automorphism associated to $\rho$ given in (ii) of the normal form theorem. The solutions of the differential equation $Ly=0$ in $\mathcal{R}$ form an $n$-dimensional $\mathcal{C}$-vector space. A basis is given by
$$
y_{\rho,i}=u_\rho^{-1}(t^\rho z^{i^*}),
$$
where $\rho$ varies over the local exponents of $L$ at $0$ and $0\leq i< m_\rho$, with $i^*=(i, \lfloor i/p\rfloor, \lfloor i/p^2\rfloor,\ldots)$.
\end{thm}
\begin{proof}
By the normal form theorem and the description of the solutions of Euler equations (Proposition \ref{prop:Euler}), we have
$$
L(y_{\rho,i})=L\circ u_\rho^{-1}(t^\rho z^{i^*})=L_0(t^\rho z^{i^*})=0,
$$
so these functions clearly are solutions of the differential equation $Ly=0$. 
Let now $y$ be any solution of $Ly=0$. Again, as $L$ commutes with the direct sum decomposition of 
$$
\mathcal{R}=\bigoplus_{\rho \in \kk}t^\rho \kk(\z)\orb x\crb
$$ 
and upon multiplication with constants of the form $x^{kp}$ we may assume that $y$ is of the form 
$y=t^\rho\left(\sum_{k=0}^\infty f_k(\z)x^k\right)
$
for $f_k\in \kk(\z)$.
If we write $L=L_0-T$ we obtain
$$
L_0y-Ty=0,
$$
where $T$ has positive shift, i.e., it strictly increases the order in $x$. Thus, $t^\rho f_0(\z)$ is a solution to the Euler equation $Ly=0$ and therefore 
$$t^\rho f_0(\z)=\sum_{(\sigma, i)}c_{\sigma, i}t^\sigma\z^{\alpha_i},$$
where $\sigma$ varies over the local exponents, $0\leq i<m_\sigma$, and $c_{\sigma, i}\in \mathcal{C}$ is homogeneous of order $0$ in $x$.
We compute
$$
L\left(y-\sum_{(\sigma, i)}c_{\sigma, i}y_{\rho, i}\right)=L\left(-\sum_{(\sigma, i)}c_{\sigma, i}u_\sigma^{-1}(t^\sigma\z^{\alpha_i})\right)=-\sum_{(\sigma, i)}c_{\sigma, i}L(u_\sigma^{-1}(t^\rho\z^{i^*}))=0.
$$
Note that for all $f\in \mathcal{F}$ we have ${\rm ord}_x (f-u(f))>{\rm ord}_x f$, i.e., the monomial of order $0$ remains unchanged under $u$. So $y-\sum_{(\sigma, i)}c_{\sigma, i}y_{\sigma, i}$ has positive order in $x$. Iteration yields constants $d_{\sigma, i}\in \mathcal{C}$ with $y=\sum_{(\sigma, i)}d_{\sigma, i}y_{\sigma,i}$. Thus, $y$ is a linear combination of $y_{\sigma, i}$. Conversely, any such linear combination is a solution of $Ly=0$. 

The linear independence of the solutions $y_{\rho, i}$ can be reduced to the linear independence of the solutions of the Euler equation, which was proven in Proposition \ref{prop:Euler}.

This proves that the solutions of $Ly=0$ in $\mathcal{R}$ form an $n$-dimensional $\mathcal{C}$-vector space with basis $y_{\rho,i}$, where $\rho$ varies over the local exponents and $0\leq i <m_\rho$.

\end{proof}

\begin{rem} \label{rem:notalgclosed}
We have assumed for convenience that our field $\kk$ is algebraically closed. If this is not the case, e.g., in the case of a finite field $\F_p$, there is no need to pass to the entire algebraic closure. In the constructions involved in the normal form theorem for an operator $L$ one has to find the roots of the characteristic polynomial $\chi_L\in \kk[s]$, the local exponents $\rho$. Further we have to evaluate the characteristic polynomial at the values $\rho+k$ for elements $k$ of the prime field of $\kk$. Thus, if $\chi_L$ splits over $\kk$ the normal form theorem works without problems within $\kk$. Otherwise it is sufficient to pass to a splitting field of $\chi_L$ to describe a full basis of solutions.
\end{rem}

\begin{rem} 
(i) The space $\mathcal{R}$ provides us with $n$ linearly independent solutions for any operator with a regular singularity at $0$ in characteristic $p$. It is minimal in the following sense: we only introduce a new variable $z_i$ whenever the algorithm constructing solutions forces us to do so, i.e., when we have to divide by $p$. It is possible to choose a system of representatives $\Lambda\subseteq\kk$ of the set $\kk/\F_p$ of residue classes and to then define
$$
\widetilde{\mathcal{R}}:=\bigoplus_{\rho \in \Lambda}t^\rho \kk(\z)\orb x\crb.
$$
It suffices to construct solutions in $\widetilde\RR$ of any linear differential equation $Ly=0$ having a regular singularity at $0$, similarly as above. For example, if $\sigma\in \kk$ is a local exponent of an Euler operator and there is $\rho\in \Lambda$ with $\rho+k=\sigma$ for some $\sigma \in \F_p$ and $k\in \F_p$, then $t^\rho x^k$ is a solution of the equation $Ly=0$. This construction has the advantage that the constants are much simpler, as they are given by the elements of
$$
\mathcal{C}_{\widetilde{\mathcal{R}}}= \kk(\z^p)\orb x^p\crb.
$$
However, this procedure requires a choice of a system of representatives of $\kk/\F_p$.

(ii) In characteristic $0$ a minimal extension of $\kk\orb x\crb$ in which every regular singular equation has a full basis of solutions is the universal Picard-Vessiot ring or field for differential equations with regular singularities, discussed in \cite{SvdP03}.
\end{rem}

\section{Examples and applications in characteristic $p$} \label{sec:appp}
\subsection{Examples}\label{subsec:exp}
\begin{ex}[Exponential function in characteristic $3$] \label{ex:ex}
We consider the equation $y'=y$. Solving over the holomorphic functions, or in $\C\osb x\csb$ one obtains the exponential function as a solution. However there is no reduction of this function modulo any prime, as all prime numbers appear in the denominators of the expansion of the exponential function. But one can obtain solutions modulo $p$  for any prime in $\mathcal{R}$ using the normal form theorem. Pick for example $p=3$. Write $L=x\partial-x=\delta-x$, so our equation is equivalent to $Ly=0$. The only local exponent of the equation is $0$, thus one needs to compute the series 
$$
\sum_{n=0}^\infty (S\circ T)^n(1).
$$
The operator $T$ is simply given by the multiplication by $x$, where $S$ is, as constructed above, a right-inverse of $L_0=x\partial$. One obtains:
\begin{align*}
(S\circ T)^1(1)&=&S(x)&=&x,\\
(S\circ T)^2(1)&=&S(x^2)&=&2x^2,\\
(S\circ T)^3(1)&=&S(2x^3)&=&2x^3z_1,\\
(S\circ T)^4(1)&=&S(2x^4z_1)&=&2x^4z_1+x^4,\\
(S\circ T)^5(1)&=&S(2x^5z_1+x^5)&=&x^5z_1,\\
(S\circ T)^6(1)&=&S(x^6z_1)&=&2x^6z_1^2,\\
(S\circ T)^7(1)&=&S(2x^7z_1^2)&=&2x^7z_1^2+2x^7z_1+x^7,\\
(S\circ T)^8(1)&=&S(2x^8z_1^2+2x^8z_1+x^8)&=&x^8z_1^2+2x^8,\\
(S\circ T)^9(1)&=&S(x^9z_1^2+2x^9)&=&x^9z_1^3z_2+2x^9z_1.
\end{align*}
One gets the solution
$$
1+x+2x^2+2x^3z_1+x^4(1+2z_1)+x^5z_1+2x^6z_1^2+x^7(1+2z_1+2z_1^2)+x^8(2+z_1^2)+x^9(2z_1+z_1^3z_2)+\ldots,
$$
which could be considered as the exponential function in characteristic $3$. Note that obtaining the  rightmost column needs some computational effort. One has to follow the steps described in Remark \ref{rem:constructive}. There seems to be no obvious pattern in the coefficients of the obtained power series.

Similarly, one can compute the exponential functions $\exp_p$ for other characteristics $p$. For $p=2$ the first terms are
$$
1+x+x^2z_1+x^3(z_1+1)+x^4(z_1^2z_2+z_1)+x^5z_1^2z_2+x^6(z_1^3z_2+z_1^3)+x^7(z_1^3z_2+z_1^2z_2+z_1^3+z_1+1)+\ldots
$$

and for $p=5$ we get
$$
1+x+3x^2+x^3+4x^4+4x^5z_1+x^6(4z_1+1)+x^7(2z_1+2)+x^8(4z_1+1)+x^9z_1+3x^{10}z_1^2+\ldots
$$

The series $\exp_p$ seems to have some remarkable properties. Let us considers the constant term in $z$, i.e. $\exp_p(x,0,0,\ldots)$. Computations suggest that $y=\exp_3(x,0,0,\ldots)$ satisfies 
\[x^3y^3 + xy^2 - y + 1=0\] and $y=\exp_5(x,0,0,\ldots)$ satisfies \[x^{10}y^5 + x^6y^4 + x^4y^3 - x^3y^3 + 2x^2y^2 + 2xy^2 - 2y + 2=0,\]
i.e. these series seem to be algebraic over $\F_p(x)$. Similar observations were made for other characteristics as well. This motivates the following challenge:
\end{ex}
\begin{problem} \label{prob:alg}
Let $L\in \F_p[x][\partial]$ be a differential operator and $\rho\in \F_p$ a local exponent of $L$. Let $u$ be the automorphism described in the normal form theorem in positive characteristic, Theorem \ref{thm:nft}. Determine the cases where $(u^{-1}(x^\rho))_{\vert z_i=0}$ is algebraic over $\F_p(x)$.
\end{problem}

This problem will be addressed in future work of the authors and Kawanoue and answered for equations of order $1$. In the next example, the answer is immediate.

\begin{ex} \label{ex:log(1-x)}
We consider the minimal complex differential equation $Ly=0$ for 
$$y(x)=-\log(1-x)=x+\frac{x^2}{2}+\frac{x^3}{3}+\ldots\in \C\osb x\csb.$$
It is given by $L=x^2\partial^2-(x^2\partial+x^3\partial^2)$. The local exponents are $0,1$ and a basis of solutions in $\C\osb x\csb$ is given by $\{1, y\}$. Reducing $L$ modulo a prime number $p$ one again finds the local exponents $0,1$. Clearly $y_{0,0}=u_0^{-1}(1)=1$. Further we compute
$$
y_{1,0}=u_1^{-1}(t^1)=\sum_{k=0}^\infty (S\circ T)(t^1)= t\left(1+\frac{x}{2}+\frac{x^2}{3}+\ldots +\frac{x^{p-2}}{p-1}+x^{p-1}z_1\right).
$$ 
Here only adjoining the variable $z_1$ instead of countably many $z_i$ is necessary to obtain enough solutions. In the next section we will describe the class of operators, where the addition of finitely many of the variables $z_i$ suffice.
\end{ex}

\subsection{Special cases} \label{subsec:specialp}

{\bf Equations with local exponents in the prime field.} The situation becomes much easier if we consider a linear differential equation $Ly=0$, whose local exponents are all contained in the prime field $\F_p\subseteq \kk$. In this case there is no need to introduce monomials $t^\rho$ with exponents $\rho\in \kk$. We define the differential subfield $\mathcal{K}$ of $\mathcal{R}$ as
$$
\mathcal{K}:=\kk(\z)\orb x\crb.
$$
One easily checks that $\mathcal{K}$ is indeed differentially closed with respect to $\6_\mathcal{R}$. Moreover, its field of constants is given by
$$\mathcal{C}_\mathcal{K}=\kk(\z^p)\orb x^p\crb.$$
The assumption on the local exponents allows one to modify the normal form theorem to use the function space 
$$\mathcal{G}^\rho:=x^\rho\sum_{k=0}^\infty \bigoplus_{\alpha \in \mathcal{A}_k}\kk \z^\alpha x^k,$$ 
instead of 
$$\mathcal{F}^\rho=t^\rho\sum_{k=0}^\infty \bigoplus_{\alpha \in \mathcal{A}_k}\kk \z^\alpha x^k,$$ 
by ``substituting $t=x$'' and analogously one obtains a full basis of solutions over $\mathcal{C}_\mathcal{K}$ in $\mathcal{K}$: For each local exponent $\rho$ one computes $u^{-1}(x^\rho)$ instead of $u^{-1}(t^\rho)$, where $u$ is the automorphism described in the normal form theorem.

{\bf Polynomial solutions.} It is well-known that if a Laurent series solution $y\in \F_p\orb x\crb$ to $Ly=0$ for an operator $L\in \F_p[x][\partial]$ with polynomial coefficients exists, then there already exists a polynomial solution to the equation, see \cite{Honda81} p.~174.  We generalize the result to solutions involving only finitely many of the variables $z_i$.

\begin{lem}\label{lem:poly}
{\it Let $\kk$ be a field of characteristic $p$. Let $L\in \kk[x][\partial]$ be a differential operator with local exponent $\rho\in \kk$. Let $y\in t^\rho \kk[z_1,\ldots, z_\ell]\osb x\csb$ be a solution of the differential equation $Ly=0$ involving only finitely many of the variables $z_i$. Let $c\in \N$. Then there exists a polynomial $q\in \kk[x, z_1,\ldots, z_\ell]$, such that $L (t^\rho q)=0$ and $y-t^\rho q\in t^\rho x^{c+1}\kk[z_1,\ldots, z_\ell]\osb x\csb$.} In particular, if a basis of power series solutions of $Ly=0$ in $\bigoplus_\rho t^\rho\kk[z_1, \ldots, z_\ell]\osb x\csb$ exists, then there already exists a basis of polynomial solutions in $\bigoplus_\rho t^\rho\kk[z_1, \ldots, z_\ell, x]$.
\end{lem}
The proof of Honda can be easily adapted to this generalisation. However, we give a more conceptual proof.

\begin{proof}
We consider $t^\rho \kk[z_1,\ldots, z_\ell] \osb x\csb$ as a free $\kk[z_1^p,\ldots, z_\ell^p]\osb x^p\csb$-module of rank $p^{\ell+1}$ with basis $\mathcal{G}=\{t^\rho x^kz^\alpha| k\in \{0,1,\ldots, p-1\}, \alpha\in \{0,1,\ldots, p-1\}^{\ell}\}$. Without loss of generality assume that $\rho=0$. We can write
$$y(x)=\sum_{g\in \mathcal{G}}y_g(z_1^p,\ldots, z_\ell^p, x^p)g$$ with series $y_g \in \kk[z_1,\ldots, z_\ell]\osb x\csb$. Then 
$$
Ly=\sum_{g\in \mathcal{G}}y_g(z_1^p,\ldots, z_\ell^p, x^p)L(g)=0
$$
implies that the series $y_g(z_1^p,\ldots, z_\ell^p, x^p)$ form a $\kk[z_1^p,\ldots, z_\ell^p]\osb x^p\csb$-linear relation between the polynomials $L(g)$ in the finite free $\kk[z_1^p,\ldots, z_\ell^p, x^p]$-module $\kk[z_1,\ldots, z_\ell, x]$ for $g\in \mathcal{G}$. By the flatness of $\kk[z_1^p,\ldots, z_\ell^p]\osb x^p\csb$ over $\kk[z_1^p,\ldots, z_\ell^p,x^p]$ there are polynomials $q_g(z_1^p,\ldots, z_\ell^p, x^p)\in \kk[z_1^p,\ldots, z_\ell^p, x^p]$ approximating $y_g(z_1^p,\ldots, z_\ell^p, x^p)$ up to any prescribed degree and such that
$$
\sum_{g\in \mathcal{G}}q_g(z_1^p,\ldots, z_\ell^p, x^p)L(g)=0. 
$$
Now set 
$$
q(z_1,\ldots, z_\ell, x)=\sum_{g\in \mathcal{G}}q_g(z_1^p,\ldots, z_\ell^p, x^p)g
$$
to get the required polynomial solution of $Ly=0$. 
\end{proof}

\begin{rem}
(i)  Assume that $L\in \kk[x][\partial]$, where $\kk$ is a {\it finite} field of characteristic $p$ with algebraic closure $\overline{\kk}$. Then if  $y\in t^\rho \overline{\kk}\osb x\csb$ is a solution obtained by the normal form theorem, we already have  $y\in t^\rho \kk(\rho)\osb x\csb$, where $\kk(\rho)$ is a finite extension of $\kk$. Recall the operators $S$ and $T$ from the normal form theorem: $S$ is a right inverse to $L_0$ and $T=L-L_0$. It holds $S(x^{\rho+k+p})=x^pS(x^{\rho +k})$ and $T(x^{\rho+k+p})=x^pT(x^{\rho+k})$. There are only finitely many $n$-tuples of elements from  $\kk(\rho)$. Write $y=t^\rho(a_0+a_1+a_2x^2+\ldots)$. Two $n$-tuples of consecutive coefficients $a_i$ of $y$, starting at powers of an index divisible by $p$, have to agree. Thus the sequence $(a_i)_{i\in \N}$ becomes periodic. Hence it suffices to take a suitable sufficiently large $k$ to obtain a polynomial solution 
$(1-x^{kp})y$ of $Ly=0$, which approximates $y$ to a prescribed degree $c$.
 
(ii) The algorithm from the normal form theorem may but need not provide us with a polynomial solution of $Ly=0$, when applied to an operator $L$ in $\F_p[x][\partial]$. To see this consider the following two examples:

(a) Let $L=x\partial-x^2\partial-x$ and $$y_L(x)=\frac{1}{1-x}$$ the solution of the equation $Ly=0$. Over $\F_p$ we compute using the algorithm from the normal form theorem with $L_0=x\partial$ and $T=x^2\partial +x$ and obtain $u^{-1}(1)=\sum_{k=0}^\infty (S_L \circ T_L)^k = 1+x+x^2+\ldots +x^{p-1}\in \F_p[x]$, a polynomial solution.

So we obtain $u^{-1}(1)=\sum_{k=0}^\infty (S_L \circ T_L)^k = 1+x+x^2+\ldots+ x^{p-1}\in \F_p[x]$, a polynomial solution.

(b) Let now $M=(-x-2x^4)+(x+x^2+-2x^4-x^5+x^7)\partial$. The equation $My=0$ is satisfied by the algebraic function $1+\frac{x}{1-x^3}$. Reducing modulo $3$ we get $$T=(x+2x^2\partial)+(2x^4\partial) + (2x^4+x^5\partial)+ (2x^7\partial)=T_1+T_3+T_4+T_6$$ and the initial form $M_0=x\partial$. We compute the solution 
$$
u^{-1}(1)=\sum_{i=0}^\infty a_ix^i\sum_{i=0}^\infty (S\circ T)^i(1)=1+x+x^4+x^7+x^{10}+\ldots
$$

Because the maximal shift of $T$ is $6$ and $(a_1, a_2, a_3, a_4, a_5, a_6)=(a_4, a_5, a_6, a_7, a_8, a_9)$ the sequence of coefficients of this series becomes periodic, as described in (i), with period length $3$.  Thus, the solution obtained by the normal form theorem in characteristic $p$ agrees with the reduction modulo $p$ of the solution obtained in characteristic $0$. 

(iii) The latter of the two examples from above illustrates that the degree of a minimal degree polynomial solution of a differential equation in characteristic $p$ need not be $p-1$, as one could expect. Indeed using the periodicity of the coefficients of the solution from above one obtains that $$\widehat y(x)= u^{-1}(1)-x^3u^{-1}(1)=1+x-x^3$$ is a polynomial solution. Any other polynomial solution has to be a multiple of $\widehat y$ with a constant. Indeed, making the ansatz 
$$
(1+x-x^3)\cdot(1+c_1x^3+c_2x^6+\cdots)=1+ax+bx^2
$$
one immediately obtains $c_1=1$, which leads to a contradiction. Therefore no polynomial solution of degree less than $3$ exists.
\end{rem}

{\bf The $p$-curvature.} Let $L$ be a differential operator. We define the {\it $p$-curvature} of $L$ to be the action of multiplication by $\partial^p$ on the space $\kk[x][\partial]/\kk[x][\partial]L$.

{\bf Operators with nilpotent $p$-curavture.} One class of operators with all local exponents in the prime field of $\kk$ turn out to be operators with nilpotent $p$-curvature.  An alternative description of these operators was provided by Honda \cite{Honda81} p.~176: We say that an equation $Ly=0$ of order $n$ has {\it sufficiently many solutions in the weak sense} if $Ly=0$ has one solution $y_1\in \kk\osb x\csb$ and recursively the equation in $u'$ of order $n-1$ obtained from $Ly=0$ by the ansatz $y=y_1u$ has sufficiently many solutions in the weak sense.
\begin{thm}[Honda, \cite{Honda81}, p.~201] \label{thm:sufficiently}
A linear differential operator $L\in \kk[x][\partial]$ has nilpotent $p$-curvature if and only if the equation $Ly=0$ has sufficiently many solutions in the weak sense.
\end{thm}
Indeed, the following theorem holds:
\begin{thm} \label{thm:primefield}
Let $L\in \kk[x][\partial]$ be a differential operator with nilpotent $p$-curvature. Then its local exponents are in the prime field $\F_p\subseteq \kk$.
\end{thm}
For a proof, see \cite{Honda81} p.~179. Further, there is another interesting characterisation of operators with nilpotent $p$-curvature due to Dwork \cite{Dwork90}. They are exactly those operators, for which finitely many of the variables $z_i$ suffice to obtain a full basis of solutions:
\begin{thm}[Dwork, \cite{Dwork90}, p.~756]
An operator $L\in \kk[x][\partial]$ has nilpotent $p$-curvature if and only if there is $l\in \N$ such that $Ly=0$ has a full basis of solutions in $\kk(z_1,\ldots, z_l)\orb x\crb$ over its field of constants $\kk(z_1^p,\ldots, z_l^p)\orb x^p\crb$.
\end{thm}

This is a generalisation of a result of Honda, who proved the result for $l=1$ and operators of order smaller than $p$, see \cite{Honda81} p.~186.

For example, the operator annihilating $\log(1-x)$, discussed in Example \ref{ex:log(1-x)}, has nilpotent $p$-curvature.

\begin{cor} \label{cor:nil}
Let $L\in \kk[x][\partial]$ be an operator with nilpotent $p$-curvature. Then there is $\ell\in \N$, such that there is a basis of polynomial solutions of $Ly=0$ in $\kk(x,z_1,\ldots, z_\ell)$.
\end{cor}

This immediately follows from Theorem~\ref{thm:primefield} and Lemma~\ref{lem:poly}. 

\subsection{The Grothendieck $p$-curvature conjecture}\label{subsec:pcurv}
We now turn to conjectures of Grothendieck-Katz, Andr\'e, B\'ezivin, Christol, the Chudnovsky brothers, Matzat and van der Put about the algebraicity of solutions of linear differential equations with polynomial coefficients defined over $\Q$ \cite{Katz70, Katz72, Katz82, Andre04, Bezivin91, Christol90, Chudnovsky83, Matzat06, vdP96}. The goal is to  study them using the normal form theorems in characteristic $0$ and $p$. 

It is a classical result, already known to Abel, that algebraic power series satisfy a linear differential equation with polynomial coefficients. The intriguing and meanwhile notorious problem is to characterize those differential equations which arise in this way, a question which appears over and over again in the literature (Abel, Riemann, Autonne, Fuchs, Frobenius, Schwarz, Beukers-Heckman, ...).

In the previous section we have studied operators with nilpotent $p$-curvature. We want to study now operators $L$ with vanishing $p$-curvature, i.e., $L$ divides $\partial^p$ from the right. The vanishing of the $p$-curvature of an operator  can be described in terms of its solutions:

\begin{lem}[Cartier] \label{lem:Car}
Let $L\in \kk[x][\partial]$ be a differential operator, where $\kk$ denotes a field of characteristic $p$. Then $L$ admits a full basis of solutions over $\kk\orb x^p\crb $ in $\kk[x]$ if and only if the $p$-curvature of $L$ vanishes.
\end{lem}

The original abstract formulation and a proof can be found in \cite{Katz70}, a more ``down-to-earth'' proof in \cite{SvdP03}. Compare this result also to Corollary \ref{cor:nil}.

In the following let $L\in\Q[x][\partial]$ be a differential operator defined over $\Q$ and denote by $L_p\in \F_p[x][\partial]$ the differential operator that arises from reducing the coefficients of $L$ modulo $p$, whenever this is defined. The reduction $L_p$ is defined for all but finitely many prime numbers $p$. We are interested in the interplay between solutions of the equations $Ly=0$ and $L_py=0$. Most prominent here is the Grothendieck $p$-curvature conjecture.

We now give an elementary formulation of the Grothendieck $p$-curvature conjecture. In this formulation the  $p$-curvature does not appear, however Cartier's Lemma \ref{lem:Car} establishes the connection.

\begin{conj}[Grothendieck $p$-curvature conjecture, \cite{Honda81}]
Let $L\in\Q[x][\partial]$. Assume that $L_py=0$ has a basis of $\F_p\osb x^p\csb$-linearly independent solutions in $\F_p\osb x\csb$ for almost all prime numbers $p$. Then there exists a basis of $\Q$-linearly independent algebraic solutions of $Ly=0$ in $\Q\osb x\csb$.
\end{conj}

\begin{rem}
One can easily generalize this conjecture to number fields, by replacing $\Q$ with $K=\Q(\alpha)$, for $\alpha$ an algebraic number, and $\F_p$ by the residue fields of $\mathcal{O}_K$ modulo its prime ideals $\mathfrak{p}$.
\end{rem}

The case of order one equations is equivalent to a special case of a theorem of Kronecker (which, in turn, is a special case of Chebotarev's density theorem) \cite{Honda81}. Katz has proven the conjecture for Picard-Fuchs equations \cite{Katz72}. There have been recent and quite technical advances in the conjecture by various people, but the general case (even for order two equations) seems to still resist. Bost has established a more general variant of the conjecture for algebraic foliations and subgroups of Lie-groups \cite{Bost01}, \cite{Chambert02}, Thm.~2.4. Progress was also made by Farb and Kisin\cite{FK09} as well as Calegari, Dimitrov and Tang \cite{CDT21}.

An apparently  weaker statement than the Grothendieck conjecture was proposed by B\'ezivin.
\begin{conj}[B\'ezivin conjecture,  \cite{Bezivin91}]
Let $L\in\Q[x][\partial]$ be a differential operator. Assume that $Ly=0$ has a basis of $\Q$-linearly independent solutions in $\Z\osb x\csb$. Then these solutions are algebraic over $\Q(x)$.
\end{conj} 

\begin{lem} \label{lem:BG}
The validity of the Grothendieck $p$-curvature conjecture implies the validity of the B\'ezivin conjecture. 
\end{lem}
In other words: The hypothesis of the B\'ezivin conjecture implies the hypothesis of the Grothendieck $p$-curvature conjecture.
 \begin{proof}
Assume that $y\in \Z\osb x\csb$ is an integral solution of $Ly=0$. Its reduction modulo all prime numbers is well-defined and a solution to $L_py=0$. For $p$ larger than the maximal difference of the local exponents of $L$, a basis of solutions of $Ly=0$ gets mapped by reduction to a basis of solution modulo $p$. The condition on $p$ is necessary to ensure that the reductions of the solutions do not become linearly dependent over $\F_p\orb x^p\crb$. Thus by the Grothendieck $p$-curvature conjecture $Ly=0$ has a basis of algebraic solutions and $y$, as a linear combination of those algebraic solutions, is algebraic itself.
\end{proof}

A substantial advance towards the Grothendieck $p$-curvature conjecture would be to prove the inverse implication of Lemma \ref{lem:BG}: in fact it would transfer the problem from positive characteristic to characteristic $0$. To approach the converse implication, it is reasonable to compare the algorithm of the normal form theorem in characteristic $0$ applied to an operator $L$ to the algorithm of the normal form theorem in characteristic $p$, applied to the reduction $L_p$ of the operator $L$ modulo $p$. We investigate in the next paragraphs how the normal form theorems could be used to achieve this. 

The problem which arises lies in the observation that the characteristic $p$ algorithm does not entirely coincide with the reduction modulo $p$ of the algorithm in zero characteristic. Very subtle disparities appear, and this makes it hard to deduce properties of the characteristic zero solutions from the characteristic $p$ solutions, in particular, to prove their algebraicity. One hope is, however, to be able to compare the Grothendieck-Katz conjecture with the B\'ezivin conjecture.

We will use the following number theoretic result:

\begin{thm}[Kronecker, \cite{Kronecker80}, Frobenius, \cite{Frobenius96}] \label{thm:density}
Let $f\in \Q[x]$ be a polynomial of degree $n$, let $s\in \N$ and $n_1,\ldots, n_s$ with $n_1+\ldots+n_s=n$. The density of prime numbers $p$ for which the reduction of $f$ modulo $p$ splits into $k$ factors of degrees $f_1,\ldots, f_k$ is equal to the number of permutations of the roots of $f$ in the Galois group of $f$ consisting of $s$ cycles of lengths $f_1, \ldots ,f_s$. In particular, $f$ splits into linear factors over $\Q[x]$ if and only if its reduction modulo $p$ splits in $\F_p[x]$ into linear factors for almost all primes $p$.
\end{thm}
This version was proven by Frobenius, while similar results were formulated by Kronecker before. It is also an easy corollary of the Chebotarev density theorem.

We now describe consequences of the hypothesis of the Grothendieck $p$-curvature conjecture. They were already collected by Honda and we refer for parts of the proof to his article. However, for the last assertion we give a different proof. It compares the two algorithms obtained from the normal form theorems in characteristics $0$ and $p$.

\def\mm{{m}}
\def \LL{{\ol L}}
\def\uu{{\ol u}}
\def\TT{{\ol T}}
\def\SS{{\ol S}}

For an operator $L\in\Q[x][\partial]$ in characteristic $0$ and a fixed prime $p$ we denote in the sequel by $\LL=L_p\in\F_p[x][\partial]$ the reduction of $L$ modulo $p$, whenever this reduction is defined.

\begin{prop}[Honda]\label{prop:nolog}
Let $L\in\Q[x][\partial]$ be a differential operator with polynomial coefficients over $\Q$. Assume that the induced equations $\LL y=0$ modulo $p$ have an $\F_p\osb x^p\csb$-basis of power series solutions in $\F_p\osb x\csb$, for almost all primes $p$. Then
\begin{enumerate}[(a)]
\item The operator $L$ has a regular singularity at $0$.
\item The local exponents of $L$ at $0$ are pairwise distinct rational numbers $\rho_i\in\Q$.
\item There exists a $\Q$-basis of Puiseux series solutions $y(x)$ of $Ly=0$ in $\sum_{\rho_i}x^{\rho_i}\Q\osb x\csb$, where $\rho_i$ ranges over the local exponents of $L$. In particular, this basis is independent of the variables $t$ and $z_i$.
\end{enumerate}
\end{prop}
\begin{proof} For part (a) use \cite{Honda81}, Corollary p.~178, combined with Theorem.~\ref{thm:sufficiently} and Lemma~\ref{lem:Car} from above.

(b) See \cite{Honda81}, Thm.~2, p.~179, combined with Thms.~\ref{thm:primefield} and \ref{thm:density} from above. We provide here a variant of Honda's proof. As a consequence of (a) there are $n$ local exponents of $L$, counted with multiplicity, $n=\ord\, L$. Moreover, for almost all primes $p$, the local exponents of $\LL $ have to be elements of the prime field $\F_p$. Indeed for any local exponent $\rho\not \in \F_p$ we obtain using Theorem \ref{thm:solp} a solution of the form $t^\rho f(x)\in t^\rho\F\osb x\csb$, contradicting the existence of a basis of $n$ solutions of $\LL y=0$ in $\F_p\osb x\csb$.

It is then shown as in \cite{Honda81} that the local exponents of $L$ are pairwise incongruent modulo almost all primes $p$. The indicial polynomial $\chi_L$ of $L$ has coefficients in $\Q$ and its reduction modulo $p$ splits into linear factors over $\F_p$ for almost all primes $p$. Thus, by Theorem \ref{thm:density}, $\chi_L$ splits into linear factors over $\Q$. It follows that all local exponents of $L$ are rational. Assume now that two local exponents are congruent modulo some $p$. Then their reduction modulo $p$ is a local exponent of $\LL $ of multiplicity at least $2$. So, Theorem \ref{thm:solp} together with the remarks in section \ref{subsec:specialp} upon avoiding the variable $t$, yield a solution of $\LL y=0$ of the form $u^{-1}(x^\rho z_1)$, where $u$ is the automorphism of the normal form theorem in positive characteristic, Theorem \ref{thm:nft}. This solution now depends on $z_1$, contradictory to the assumption. This proves (b).

(c) By Theorem \ref{thm:sol0} a basis of solutions of $Ly=0$ lies in 
$$
\sum_{\rho_i}x^{\rho_i}\Q\osb x\csb[z],
$$
the sum varying over all local exponents $\rho_i$ of $L$. It remains to prove that these solutions are independent of $z$. So assume the contrary: let $f$ be a solution which depends on $z$. Without loss of generality we may assume that 
$$
f=u^{-1}(x^\rho)=x^\rho(1+a_1(z)x+a_2(z)x^2+\ldots)
$$ 
for some local exponent $\rho\in\Q$ of $L$ and some $a_i\in \Q[z]$. Let $\mm\in \N$ be the first index where $a_\mm$ depends on $z$. We will construct from $f$ a solution $g$ of $\LL y=0$, for a suitable prime $p$, which involves $z_1$-terms which are not $p$-th powers. This will produce the required contradiction.

The construction of $g$ is, in fact, quite subtle. We have to run the two normal form algorithms for the construction of $f$ and $g$ simultaneously in characteristic $0$ and $p$ as long as no $z$ appears in characteristic $0$. At the moment when $z$ occurs for the first time, say, in the computation of the coefficient $a_\mm$ of $f$, a careful comparison ensures  that $z_1$ shows up also in the expansion of $g$ in the characteristic $p$ algorithm.

We choose the prime $p$ subject to the following conditions:
\begin{itemize}
\item $p> n=\ord\,L$;
\item There is a basis of solutions of $\LL y=0$ in $\F_p\osb x\csb$;
\item $p$ does not divide any of the denominators of the local exponents of $L$;
\item $p$ does not divide any of the denominators of the coefficients of $a_1,\ldots, a_k$.
\end{itemize}

Let $\Lambda$ be the set of positive integers $\ell $ smaller than $\mm $ such that $\sigma:=\rho+\ell $ is a local exponent of $\LL $. Here we write $\sigma$ and $\rho$ for elements in $\Q$ as well as for the representatives in $\{0,1,\ldots, p-1\}$ of their reduction modulo $p$. We define 
$$
g=\uu ^{-1}\left(x^\rho\left(1+\sum_{\ell \in \Lambda}\overline{a_\ell }x^\ell \right)\right)=x^\rho(1+b_1(t,z)x+b_2(t,z)x^2+\ldots),
$$
where $\uu $ is the automorphism of $\mathcal{G}^\rho=x^\rho\sum_{k=0}^\infty \bigoplus_{\alpha \in \mathcal{A}_k}\kk \z^\alpha x^k$  from the normal form theorem, Theorem \ref{thm:nft}, compare again to the remarks on avoiding the variable $t$ in Section \ref{subsec:specialp}. The additional summand $\sum_{\ell \in \Lambda}\overline{a_\ell }x^\ell$ in the inner parenthesis of $g$ is required to make $f$ and $g$ coincide up to degree $\mm-1$.

We will show that $g$ is a solution of $\LL y=0$ and that its coefficient $b_\mm$ involves $z_1$. The first thing is easy since, by the normal form theorem \ref{thm:nft}, 
$$
\LL(g)=(\LL_0\circ\uu)(g)=\LL_0\left(x^\rho\left(1+\sum_{\ell \in \Lambda}\overline{a_\ell }x^\ell\right)\right)=0,
$$
for $\rho+\ell$ is a local exponent of $\LL_0$ and hence $\LL_0(x^{\rho+\ell})=0$. This proves $\LL g=0$.

%---------------------------------

Next we prove inductively that $b_\ell=\overline{a_\ell}$ for $\ell\leq\mm-1$, i.e., that the expansion of $g$ up to degree $\mm-1$ equals the reduction of the respective expansion of $f$. This part is a bit computational.

Write $T=L-L_0$ and $\TT=\LL-\LL_0$ as earlier for the tails of $L$ and $\LL$. We expand $T$ and $\TT$ as sums of Euler operators
$$
T=T_1+\cdots+T_r,
$$
$$
\TT=\TT_1+\cdots+\TT_r,
$$
Similarly, we define $S$ and $\SS$ as the inverses $S=(L_0\vert_{\mathcal{H}})^{-1}$ and $\SS=(\LL_0\vert_{\mathcal{\ol H}})^{-1}$ of $L_0$ and $\LL_0$, respectively, on direct complements of their kernels, as described in the normal form theorems, Theorems~\ref{thm:nft0} and \ref{thm:nft}.

We now distinguish two cases.

%---------------------------------

(i) Assume first that $\rho+\ell $ is not a local exponent of $L$. Rewriting the differential equations $Ly=0$ and $\LL y=0$ as linear recursions for the coefficients of the prospective solutions we obtain, for $\ell\leq m-1$,
$$
a_{\ell }=S\left(\sum_{k=1}^r T_k\left(a_{\ell -k}x^{\rho+\ell -k}\right)\right),
$$
and
$$
b_{\ell }=\SS\left(\sum_{k=1}^r \TT_k\left(b_{\ell -k}x^{\rho+\ell -k}\right)\right),
$$
where both sums in the parentheses are homogeneous of degree $\rho+\ell $ in $x$. By induction on $\ell$ we may assume that $b_{\ell-k}=\ol{a_{\ell-k}}$ equals the reduction of $a_{\ell-k}$ for all $k=1,\ldots,r$. Hence this also holds for $b_\ell=\ol{\a_\ell}$.

%---------------------------------

(ii) Assume now that $\rho+\ell $ is a local exponent of $L$. Here, the formula for $b_\ell$ is different, by the very definition of $g$,
$$
b_{\ell }=\SS\left(\sum_{k=1}^r\TT_k(b_{\ell -k}x^{\rho+\ell -k})\right) + \overline{a_{\ell }}.
$$
Now, as $\rho+\ell$ is a local exponent of $L$ and hence also of $\LL$, the image $\SS(x^{\rho+\ell})$ will involve $z_1$. Therefore, as $b_\ell$ does not involve $z_1$ by assumption, we get $\TT_k(b_{\ell -k}x^{\rho+\ell -k})=0$. Hence again $b_\ell=\ol{a_{\ell}}$ for all $\ell\leq \mm-1$. 

This proves in both cases that $g$ is the reduction of $f$ modulo $p$ up to degree $\mm-1$. 

%---------------------------------

To finish the proof we will show that $b_\mm$ involves $z_1$. This will produce the required contradiction. As $a_\mm$ depends on $z$ by assumption, $\rho+\mm$ is necessarily a local exponent of $L$ and thus of $\LL $. Hence $\SS(x^{\rho+\mm})$ will depend on $z_1$. Recall that
$$
a_\mm=S\left(\sum_{k=1}^r T_k(a_{\mm -k}x^{\rho+\mm -k})\right),
$$
and
$$
b_\mm=\SS\left(\sum_{k=1}^r \TT_k(b_{\mm-k}x^{\rho+\mm-k})\right).
$$

From the already established equalities $b_\ell=\ol{a_\ell}$ for $\ell\leq \mm-1$ it follows that $\TT_k(b_{\mm-k}x^{\rho+\mm-k})$ is the reduction modulo $p$ of $T_k(a_{\mm -k}x^{\rho+\mm -k})$. Now, if $T_k(a_{\mm -k}x^{\rho+\mm -k})$ were $0$, its image $a_\mm$ under $S$ were zero, which is excluded by the choice of $\mm$. So this term is non-zero. But then it suffices to choose $p$ sufficiently large such that also the reduction $\TT_k(b_{\mm-k}x^{\rho+\mm-k})$ is non-zero. Similarly as before we then get that $b_\mm$ involves $z_1$, contradiction.
\end{proof}

We illustrate the crucial step in the proof of (c) by an example.

\begin{ex}
The operator $L=x^2\partial^2-3x\partial-3x-x^2-x^3$ has the solution
$$
f(x)=u^{-1}(1)=1+a_1x+a_2x^2+\ldots = 1-x+\frac{1}{2}x^2-\frac{1}{2}x^3-\frac{1}{2}x^4z+\ldots,
$$
so  $a_4=-\frac{1}{2}z$ is the first coefficient which depends on $z$. Assume that there was a full basis of solutions in $\F_3\osb x\csb$. The local exponents in characteristic $3$ are $0$ and $1$, so $\Omega_4=\{1, 3\}$. We compute, using $\ol T=x^2+x^3$ and $\ol L_0=x^2\partial^2$, the expansion of the following solution
$$\ol u_3^{-1}(1+2x+x^3)=1+2x+2x^2+x^3+\ldots,$$
which agrees with the reduction of $f$ up to order $3$. However, the next term in the expansion is $S_3(x^4)=x^4z_1$, so $u_3^{-1}(1+2x+x^3)\not \in \F_3\osb x\csb$.
\end{ex}

\subsection{Outlook}\label{subsec:outlook}

If one wants to pursue the goal of proving the equivalence of the Grothendieck $p$-curvature conjecture and the B\'ezivin conjecture, number theoretic obstacles occur.

A power series $y(x)\in \Q\osb x\csb$ is called {\it globally bounded} if there is an integer $N$
 such that $y(Nx)\in \Z\osb x\csb$. In other words, there are only finitely many prime numbers $p$ appearing in the denominators of the coefficients of $y$ and they only grow geometrically. A theorem of Eisenstein \cite{Eisenstein52} says that any algebraic power series is globally bounded. 

To prove that the validity of the B\'ezivin conjecture implies the validity of the Grothendieck $p$-curvature conjecture it suffices to show that for a linear differential equation $Ly=0$ whose reduction $L_py=0$ has a full basis of solutions in $\F_p\osb x\csb$ the basis of solutions in characteristic $0$ is globally bounded. For this it is natural to try to compare the algorithms from the normal form theorems in characteristic $0$ and $p$ further. Ideally, $p$ would not appear in the denominators of solutions in characteristic $p$ if and only if there is a basis of solutions in $\F_p\osb x\csb$ of $L_py=0$, at least for almost all $p$. However, the situation is not as easy as one might hope, as the following two examples illustrate:

\begin{ex} \label{ex:GB}
(i) The first example shows that for finitely many primes it may happen that a full basis of solutions of the reduction of a linear differential equation modulo $p$ exists, although $p$ appears in the denominator of one of the solutions in characteristic $0$. The solution of $\partial - nx^{n-1}$ for $n\in \N$ is $e^{x^n}$, a power series where each prime number appears eventually in the denominators. However, for all prime numbers $p$ dividing $n$, the reduction of the equation modulo $p$ is an Euler equation having the solution $1\in \F_p\osb x\csb$. As this can happen only for a finite number of primes, this does not contradict the Grothendieck $p$-curvature conjecture.

(ii) The next example shows that to rule out the appearance of the prime factor $p$ in the denominators of a solution of $Ly=0$ it is not sufficient to work on the level of individual solutions associated to a local exponent and its reduction. If possible at all, it has to take into account the existence of a full basis of solutions.

The power series $$y(x)=\sum_{k=1}^\infty a_kx^k=\sum_{k=1}^\infty \frac{k(k+2)}{(k+1)}x^k=\frac{3}{2}x+\frac{8}{3}x^2+\frac{15}{4}x^3+\frac{24}{5}x^4+\ldots=\frac{\log(1-x)}{x}+\frac{x}{(x-1)^2}$$
is annihilated by the third order operator
$$ L=x^3\partial^3+4x^2\partial^2+x\partial-1-(x^4\partial^3+8x^3\partial^2+13x^2\partial+3x).$$ 
This operator $L$ is {\it hypergeometric}, i.e., $T=L-L_0$ is an Euler operator with shift one. Moreover, $y$ is annihilated by the second order operator
$$
M=3x^2\partial^2+3x\partial-3+(x^4-4x^3)\partial^2+(3x-12x^2)\partial+x^2-4x,
$$
which is not hypergeometric. The operator $M$ is a right divisor of $L$, as one verifies that
$$
\left(-\frac{1}{x-3}x\partial-\frac{1}{x-3}\right)M=L.
$$

Let us first concern ourselves with the operator $L$. Its local exponents are $-1$ with multiplicity two and $1$ with multiplicity $1$.  We have $y=\frac{3}{2}\cdot u^{-1}(x)$, where $u$ is the automorphism described in the normal form theorem in characteristic $0$. Moreover we compute $u^{-1}(x^{-1})=x^{-1}$ and $u^{-1}(x^{-1}z)=x^{-1}z$. Thus a basis of solutions of $Ly=0$ is given by  $y, x^{-1}$ and $x^{-1}\log x $.

For all prime numbers $p$ the coefficient of $x^{p-2}$ in the expansion of $y$ is divisible by $p$, while the denominators of $a_1,\ldots, a_{p-2}$ are not. Thus 
$$
y_p:=\sum_{k=1}^{p-2}a_kx^k
$$
is well defined in characteristic $p$ and a solution to the equation $L_py=0$. It is given as $u^{-1}_p(x)$ where $u_p$ is the automorphism defined in the normal form theorem in characteristic $p$. The series $y$ is not algebraic, as it is not globally bounded. In fact any prime number $p$ appears in the denominators of the coefficients $a_i$. However, the solution in characteristic $p$ corresponding to the reduction of the local exponent $1$ is a genuine power series. Other linearly independent solutions in characteristic $p$ are $x^{-1}$ and $x^{-1}z_1$. We see that in neither characteristic there is a basis of power series solutions.

Let us now turn to the operator $M$, which has local exponents $-1$ and $1$ as well, both with multiplicity $1$. A basis of solutions is given by $x^{-1}$ and $y$. This does not contradict the Grothendieck $p$-curvature conjecture, as $y_p$ is not a solution of $M$. For $L$ the construction was very dependent on the fact that the equation is hypergeometric, which is no longer the case for $M$.
\end{ex}

There still remain several questions about linear differential equations over fields with positive characteristic. For linear differential equations with holomorphic coefficients there is a criterion by Fuchs characterizing regular singular points of an operator $L$ \cite{Fuchs66}. A point $a\in \mathbb{P}^1_\C$  is at most a regular singularity of $L$ if and only if there is a local basis of solutions of $Ly=0$, which grows at most polynomially when approaching $a$. One would expect a similar criterion in characteristic $p$: an $n$-dimensional vector space of solutions in $\mathcal{R}$ over the constants $\mathcal{C}$ should suffice to conclude that $0$ is a regular singular point of $L$. The needed framework could be provided by adapting the solution theory of linear differential equations with holomorphic coefficients and an irregular singularitiy at $0$ of N. Merkl, described in section \ref{subsec:app0} to positive characteristic. More precisely, this should allow the definition of a ring $\widetilde{\mathcal{R}}$ in which every linear differential equations with an irregular singularity at $0$ in positive characteristic has a basis of solutions. The corresponding criterion in characteristic $p$ should then read: A linear differential equations $Ly=0$ admits a basis of solutions in $\mathcal{R} \subseteq \widetilde{\mathcal{R}}$ if and only if $0$ is a regular singularity of $L$

Moreover, the solutions of differential equations in $\mathcal{R}$ need to be better understood. For example one would expect some kind of pattern in the exponential function  in positive characteristic discussed in Example \ref{ex:ex}. However, no such structure seems obvious. Also the question of the algebraicity of the constant term raised in Example \ref{ex:ex} and Problem \ref{prob:alg} deserves some attention and should be studied for the constant terms of solutions of any equation.

In addition there is hope to extract information about the $p$-curvature of linear differential operator $L$ in positive characteristics from the description of a full basis of solutions in the differential extension $\mathcal{R}$ of $\kk$. In \cite{BCS15} Bostan, Caruso and Schost describe an algorithm on how to effectively compute the $p$-curvature of a differential operator. They work over the ring $\kk\osb x\csb^{\rm dp}$ of series of the form
$$
f=a_0+a_1\gamma_1(x)+a_2\gamma_2(x)+\ldots.
$$
The elements $\gamma_i(x)$ are formal variables, but should be thought of $\frac{x^i}{i!}$. The multiplication on $\kk\osb x\csb^{\rm dp}$ is consequently given by $\gamma_{i}(x)\gamma_j(x)=\binom{i+j}{i}\gamma_{i+j}(x)$. In a suitable extension extension of $\kk\osb x\csb^{\rm dp}$, accounting for local exponents outside the prime field, they construct a basis of solutions of $Ly=0$. From this basis they compute, passing to systems of first order equations, the matrix representation of the $p$-curvature. A similar program seems feasible working in $\mathcal{R}$ instead of $\kk\osb x\csb^{\rm dp}$.

Finally, there remain, of course, the Grothendieck $p$-curvature conjecture and the B\'ezivin conjecture. As Example \ref{ex:GB} shows, the algorithms of the normal form theorems in characteristic $p$ and $0$ show some unexpected discrepancy. The hope that solutions of the reduction of differential operators are reductions of solutions of the operator seems to be unfounded. However, the phenomena shown require further investigation.

\addcontentsline{toc}{section}{References}
\printbibliography
\textsc{University of Vienna, Faculty of Mathematics, Oskar-Morgenstern-Platz 1, 1090, Vienna, Austria}\\
\textit{Email: }\href{mailto:florian.fuernsinn@univie.ac.at}{\texttt{florian.fuernsinn@univie.ac.at}}\\
\textit{Email: }\href{mailto:herwig.hauser@univie.ac.at}{\texttt{herwig.hauser@univie.ac.at}}

\end{document}